\documentclass[11pt, a4paper]{amsart}

\usepackage{amsmath,amssymb,amsthm}
\usepackage{bbm,mathrsfs}
\usepackage[all]{xy}

\makeatletter
 \renewcommand{\theequation}{\arabic{section}.\arabic{equation}}
  \@addtoreset{equation}{section}
\makeatother

\newcommand{\cmdblackltimes}{\mathbin{\raisebox{0.2ex}{
\makebox[0.92em][l]{${\scriptstyle\blacktriangleright\mathrel{\mkern-4mu}<}$}}}}

\usepackage[mathscr]{eucal}

\newcommand{\cotens}{\mathbin{\Box}}
\newcommand{\tens}{\otimes}

\newcommand{\id}{\mathsf{id}}

\newcommand{\mathc}{\mathbb{C}}
\newcommand{\mathq}{\mathbb{Q}}
\newcommand{\mathz}{\mathbb{Z}}
\newcommand{\mathr}{\mathbb{R}}
\newcommand{\Deltaa}{\triangle}

\newcommand{\eve}{{\bar{0}}}
\newcommand{\odd}{{\bar{1}}}

\newcommand{\G}{\mathbbmss{G}}
\newcommand{\U}{\mathbbmss{U}}
\newcommand{\B}{\mathbbmss{B}}
\newcommand{\T}{\mathbbmss{T}}
\renewcommand{\H}{\mathbbmss{H}}
\renewcommand{\P}{\mathbbmss{P}}
\newcommand{\K}{\mathbbmss{K}}

\newcommand{\s}{\mathsf{S}}

\newcommand{\ev}{{\mathsf{ev}}}
\newcommand{\Gev}{{\G_{\ev}}}
\newcommand{\Uev}{{\U_{\ev}}}
\newcommand{\Bev}{{\B_{\ev}}}
\newcommand{\Tev}{{\T_{\ev}}}
\newcommand{\Pev}{{\P_{\ev}}}
\newcommand{\Kev}{{\K_{\ev}}}

\newcommand{\Uevp}{{\U^+_{\ev}}}

\newcommand{\W}{\mathsf{W}}

\newcommand{\dominant}{\flat}

\newcommand{\GL}{\mathbbmss{GL}}
\newcommand{\GLL}{\mathrm{GL}}

\newcommand{\SL}{\mathbbmss{SL}}
\newcommand{\SpO}{\mathbbmss{SpO}}
\newcommand{\Q}{\mathbf{Q}}

\newcommand{\Ga}{G_{\mathrm{add}}}

\newcommand{\UU}{\mathcal{U}}

\newcommand{\X}{\mathsf{X}}

\newcommand{\MM}{\mathtt{M}}
\newcommand{\QQ}{\mathtt{Q}}

\newcommand{\g}{\mathfrak{g}}
\newcommand{\h}{\mathfrak{h}}
\newcommand{\n}{\mathfrak{n}}
\newcommand{\bb}{\mathfrak{b}}

\newcommand{\uu}{\mathfrak{u}}
\newcommand{\q}{\mathfrak{q}}
\newcommand{\gl}{\mathfrak{gl}}

\newcommand{\e}{\mathsf{e}}

\newcommand{\hy}{\mathsf{hy}}
\newcommand{\Dist}{\mathsf{Dist}}
\newcommand{\Lie}{\mathsf{Lie}}
\newcommand{\Ker}{\mathsf{Ker}}

\newcommand{\hits}{\rightharpoonup}

\newcommand{\Mat}{\mathsf{Mat}}

\newcommand{\gr}{\mathsf{gr}}

\newcommand{\Frac}{\mathsf{Frac}}
\newcommand{\Stab}{\mathsf{Stab}}
\newcommand{\Inv}{\mathsf{Inv}}

\newcommand{\spl}{\mathsf{spl}}

\newcommand{\Simple}{\mathsf{Simple}}

\newcommand{\res}{\mathsf{res}}
\newcommand{\ind}{\mathsf{ind}}
\newcommand{\coind}{\mathsf{coind}}
\newcommand{\rad}{\mathsf{rad}}
\newcommand{\soc}{\mathsf{soc}}
\newcommand{\corad}{\mathsf{corad}}
\newcommand{\Ind}{\mathsf{Ind}}

\newcommand{\coad}{\mathsf{coad}}

\newcommand{\rank}{\mathrm{rank}}

\newcommand{\ch}{\mathsf{ch}}
\newcommand{\co}{\mathsf{co}}

\newcommand{\charr}{\mathsf{char}}

\newcommand{\C}{\mathsf{Cl}}

\newcommand{\Mod}{\mathsf{Mod}}
\newcommand{\SMod}{\mathsf{SMod}}

\newcommand{\Aut}{{\mathsf{Aut}}}

\newcommand{\Cat}{{\mathcal{C}}}

\newtheorem{thm}{Theorem}[section]
\newtheorem{lemm}[thm]{Lemma}
\newtheorem{prop}[thm]{Proposition}
\newtheorem{cor}[thm]{Corollary}

\theoremstyle{definition}
\newtheorem{deff}[thm]{Definition}
\newtheorem{rem}[thm]{Remark}
\newtheorem{ex}[thm]{Example}
\newtheorem{exs}[thm]{Examples}

\begin{document}

\title[Borel-Weil Theorem for Algebraic Supergroups]{Borel-Weil Theorem for Algebraic Supergroups}

\author[T.~Shibata]{Taiki Shibata}
\address[T.~Shibata]{Department of Applied Mathematics,
Okayama University of Science,
1-1 Ridai-cho Kita-ku Okayama-shi, Okayama 700-0005, Japan}
\email{shibata@xmath.ous.ac.jp}

\date{\today}

\maketitle

\begin{abstract}
We study the structure of an algebraic supergroup $\G$
and establish the Borel-Weil theorem for $\G$
to give a systematic construction of 
all simple supermodules over an arbitrary field.
Especially when $\G$ has a distinguished parabolic super-subgroup,
we show that the set of all simple supermodules of $\G$ is 
parameterized by the set of all dominant weights for the even part of $\G$, 
prove a super-analogue of the Kempf vanishing theorem, and 
give a description of Euler characteristics.
\end{abstract}



\section{Introduction}\label{sec:intro}
This paper aims to provide a framework of characteristic-free study of
representation theory of algebraic supergroups,
and to prove some new results including a super-analogue of the Borel-Weil theorem.
An algebraic supergroup is a representable group-valued functor
$\G$ defined on the category of commutative superalgebras,
such that the representing (necessarily, Hopf) superalgebra
$\mathcal{O}(\G)$ is finitely generated.
Here, the word ``super'' is a synonym of 
``graded by the group $\mathz_2=\{\eve,\odd\}$ of order two''.
There exists the largest ordinary ($=$ non-super) subgroup $\Gev$ of $\G$, 
called the {\it even part} of $\G$.
Since the notion of algebraic supergroups
is a generalization of algebraic groups,
it is natural to ask 
what is a good/nice generalization of 
reductive groups to our super setting.
In 2011, V. Serganova \cite{Ser11} answered this question,
introducing the notion of {\it quasireductive supergroups};
it is an algebraic supergroup 
whose even-part $\Gev$ is a split and connected reductive group.
There are plenty of examples of quasireductive supergroups;
general linear supergroups $\GL(m|n)$,
queer supergroups $\Q(n)$ (whose Lie superalgebra is the queer Lie superalgebra $\q(n)$),
and
all Chevalley supergroups of classical type 
(which is a super-analogue of the Chevalley-Demazure groups,
including special linear supergroups $\SL(m|n)$ and 
ortho-symplectic supergroups $\SpO(m|n)$)
due to R.~Fioresi and F.~Gavarini \cite{FioGav12,FioGav13} in 2012.

Over a field of characteristic zero,
the study of quasireductive supergroups $\G$ is 
essentially the same as its (necessarily, finite-dimensional) 
Lie superalgebra $\Lie(\G)$ of $\G$.
Ever since V.~Kac \cite{Kac77} classified
finite-dimensional simple Lie superalgebras over $\mathc$,
representations of $\Lie(\G)$ has been well-studied,
see V.~Kac \cite{Kac77-3}, I.~Penkov and V.~Serganova \cite{PenSer92}, and
A.~Sergeev \cite{Serg01} for example.
See also I.~Musson \cite{Mus12}, S.-J.~Cheng and W.~Wang \cite{CheWan12}.
Serganova \cite{Ser11}
gave a systematic construction of 
simple supermodules for quasireductive supergroups over 
an algebraically closed field of characteristic zero,
in terms of their Lie superalgebras.

On the other hand, 
the study of quasireductive supergroups
over a field of positive characteristic has just started.
A.~Zubkov and F.~Marko \cite{Mar18,MarZub18,Zub06,Zub16}
have obtained may results for representations of $\GL(m|n)$.
J.~Brundan and A.~Kleshchev \cite{BruKle03} studied representations of $\Q(n)$
which have a close relationship to modular representations of spin symmetric groups.
By using representations of $\GL(m|n)$,
Mullineux conjecture (now is Mullineux theorem) was re-proven 
by J.~Brundan and J.~Kujawa \cite{BruKuj03}.
B.~Shu and Wang \cite{ShuWan08} classified simple supermodules of $\SpO(m|n)$
whose parameter set is described by combinatorics terms,
which is related to Mullineux bijection.
Recently,
S.-J.~Cheng, B.~Shu and W.~Wang \cite{CheShuWan18} explicitly classified
all simple supermodules of (simply connected) Chevalley groups of type
$D(2|1;\zeta),G(3)$ and $F(3|1)$.

This paper proposes a framework of characteristic-free study of
representation theory of quasireductive supergroups $\G$.
The first purpose of the paper is to establish
the Borel-Weil theorem for $\G$.
Namely,
we give a systematic construction of all simple $\G$-supermodules
which extends Serganova's construction to arbitrary characteristic.
In the analytic situation,
the Borel-Weil-Bott theorem for classical Lie supergroups over $\mathc$,
was established by I.~Penkov \cite{Pen88-1} in 1990.
Recently, A.~Zubkov \cite{Zub16} proved the Borel-Weil-Bott theorem
for $\G=\GL(m|n)$ over a field.
In the following,
we explain the details of our approach.
First, 
using a concept of 
the theory of Harish-Chandra pairs due to A.~Masuoka \cite{Mas12},
we are able to find
a super-torus $\T$ and a Borel super-subgroup $\B$ of $\G$.
Next, we construct all simple $\T$-supermodules $\{\uu(\lambda)\}_{\lambda\in\Lambda}$ 
by the theory of Clifford algebras,
where $\Lambda$ is the character group of $\Tev$.
For the quotient superscheme $\G/\B$ (due to A.~Masuoka and A.~Zubkov \cite{MasZub11}),
we may consider cohomology groups 
$H^n(\lambda):= H^n(\G/\B,\mathscr{L}(\uu(\lambda)))$,
where $\lambda\in\Lambda$, $n\geq0$ and
$\mathscr{L}(\uu(\lambda))$ is the associated sheaf to $\uu(\lambda)$ on $\G/\B$.
Set $\Lambda^\dominant:=\{\lambda\in\Lambda \mid H^0(\lambda)\not=0\}$.
Our main result is the following.

\medskip

\noindent{\bf Theorem (Theorem~\ref{prp:LMQ}).}
{\it
For $\lambda\in\Lambda^\dominant$, 
the $\G$-socle $L(\lambda)$ of $H^0(\lambda)$ is simple.
The map $\lambda \mapsto L(\lambda)$ 
gives a one-to-one correspondence
between $\Lambda^\dominant$ and 
the isomorphism classes of all simple $\G$-supermodules up to parity change.
}

\medskip

The second purpose of the paper is to 
study quasireductive supergroups $\G$ which include
a distinguished ``parabolic'' super-subgroup $\P$ such that
$\Pev=\Gev$ and $\Lie(\P)_\odd=\Lie(\B)_\odd$.
Then we can describe $H^n(\lambda)$ in terms of
the non-super cohomology group 
$H^n_\ev(\lambda) := H^n(\Gev/\Bev,\mathscr{L}(\Bbbk^\lambda))$ as $\Tev$-supermodules,
where $\Bbbk^\lambda$ is the one-dimensional $\Tev$-module of weight $\lambda$.
This is a generalization of Zubkov's result 
\cite[Proposition~5.2 and Lemma~5.1]{Zub06} for $\G=\GL(m|n)$,
see also Marko \cite[\S1.2]{Mar18}.
As a corollary, we have the following results.

\medskip

\noindent{\bf Theorem (Section~\ref{sec:final-appl}).}
{\it
If $\G$ has such a parabolic super-subgroup $\P$,
then
\begin{enumerate}
\item
$\Lambda^\dominant$ coincides with 
the set of all dominant weights $\Lambda^+$ for $\Gev$,
\item
$H^n(\lambda)=0$ for $\lambda\in\Lambda^+$, $n\geq1$, and
\item
the formal character of $H^0(\lambda)$
is given by $\ch(H^0_\ev(\lambda)) \prod_{\delta\in\Delta^+_\odd}(1+\e^{-\delta})$
for $\lambda\in\Lambda^+$,
where $\Delta^+_\odd$ is the set of all positive odd roots of $\Lie(\G)$.
\end{enumerate}
}

\smallskip

\noindent{The} above (2) is a super-analogue of the Kempf vanishing theorem.

\medskip

The paper is organized as follows.
In Section~\ref{sec:prelim}, 
we review some basic definitions and results for algebraic supergroups
over a principal ideal domain (PID).
As a generalization of \cite[Definition~6.3]{MasShi17},
we re-define the notion of {\it sub-pairs} (Definition~\ref{def:sub-pair}).
In Section~\ref{sec:Quasired},
we give a definition of {\it quasireductive supergroups} over $\mathz$ 
(Definition~\ref{def:Quasired})
as a generalization of Serganova's one.
We study structure of quasireductive supergroups $\G$ over a PID;
super-tori $\T$, unipotent super-subgroups $\U,\U^+$ and 
Borel super-subgroups $\B,\B^+$ are given 
in \S\ref{sec:tori,unip} and \S\ref{sec:Borel}.
In particular, we give a PBW basis of Kostant's $\mathz$-form for $\G$ 
(Theorem~\ref{prp:PBW}).
As a consequence,
we prove an algebraic interpretation of
the ``density'' of the big cell $\B\times \B^+$ in $\G$
(Corollary~\ref{prp:inj}).
This is needed to prove the Borel-Weil theorem for our $\G$.
In the non-super situation, see 
B.~Parshall and J.~Wang \cite{ParWan91}, and 
J.~Bichon and S.~Riche \cite{BicRic16}.

In Section~\ref{sec:simples},
we work over a field $\Bbbk$ of characteristic not equal to $2$.
First, in \S\ref{sec:irrT}, we construct all simples $\{\uu(\lambda)\}_{\lambda\in\Lambda}$
of a super-torus $\T$,
using the theory of Clifford algebras (Appendix~\ref{sec:clif}).
Here, $\Lambda$ is the character group $\X(\Tev)$ of $\Tev$.
When $\Bbbk$ is algebraically closed,
we show that our simples are the same as Serganova's
(Remark~\ref{rem:Serganova}).
Next two subsections are devoted to 
cohomology groups $H^n(\lambda)$
and determination of all simple $\B$-supermodules.
In \S\ref{sec:irrG},
we prove our first main result,
that is,
the $\G$-socle $L(\lambda)$ of $H^0(\lambda)\not=0$ is simple,
and this gives a one-to-one correspondence between
the set $\Lambda^\dominant:=\{\lambda\in \Lambda \mid H^0(\lambda)\not=0 \}$
and the set of all isomorphism classes $\Simple_\Pi(\G)$ of 
simple $\G$-supermodules up to parity change.
We also give a necessary and sufficient condition
for $L(\lambda)$ to be of type $\QQ$,
that is,
$L(\lambda)\cong \Pi L(\lambda)$ as $\G$-supermodules,
using the language of Clifford algebras (Theorem~\ref{prp:LMQ}).
Here, $\Pi$ is the parity change functor.
Some further properties of the induced supermodule $H^0(\lambda)$ 
are discussed in \S\ref{sec:someH^0}.

In the final Section~\ref{sec:parab}, we still work over a field $\Bbbk$
of characteristic not equal to $2$.
A.~Zubkov \cite{Zub06} establish various results on $\GL(m|n)$,
using its ``parabolic'' super-subgroups.
In this section, we assume that our quasireductive supergroup $\G$
having a parabolic super-subgroup $\P$ such that 
$\Pev=\Gev$ and $\Lie(\P)_\odd=\Lie(\B)_\odd$.
In \S\ref{sec:tens-split},
we generalize some of results stated in \cite[\S5]{Zub06}.
In \S\ref{sec:final-appl}, 
we prove that for all $\lambda\in\Lambda^\dominant$ and $n\geq0$,
$H^n(\lambda)$ is isomorphic to 
$H^n_\ev(\lambda) \tens \wedge \Lie(\U^+)_\odd^*$
as $\Tev$-supermodules (Theorem~\ref{prp:main}),
where $H^n_\ev(\lambda)$ is the cohomology group for $\Gev$
and $\Lie(\U^+)_\odd^*$ is the dual space of the odd part of
$\Lie(\U^+)$.
As a consequence,
we get the following three results: 
(1) $\Lambda^\dominant$ coincides with 
the set of all dominant weights $\Lambda^+$ for $\Gev$,
(2) the Kempf vanishing theorem $H^n(\lambda)=0$ for 
$\lambda\in\Lambda^+$, $n\geq1$ (Corollary~\ref{prp:cor1}),
and 
(3) the Weyl character formula, that is,
an explicit description of 
the formal character of $H^0(\lambda)$ for $\lambda\in\Lambda^+$ (Corollary~\ref{prp:cor2}).
In \S\ref{sec:ex}, we see some examples.

\section{Preliminaries} \label{sec:prelim}

In this section, the base ring
$\Bbbk$ supposed to be a principal ideal domain (PID) of
characteristic different from $2$.
The unadorned $\tens$ denoted the tensor product over $\Bbbk$.
Note that, a $\Bbbk$-module is projective if and only if it is free.
A $\Bbbk$-module is said to be {\it finite} (resp. {\it $\Bbbk$-free/$\Bbbk$-flat}) 
if it is finitely generated (resp. free/flat).
In particular, we say that a $\Bbbk$-module is {\it finite free} 
if it is finite and $\Bbbk$-free.
All $\Bbbk$-modules form a symmetric tensor category $\Mod$
with the trivial symmetry $V\tens W \to W\tens V$; $v\tens w \mapsto w\tens v$,
where $V,W\in\Mod$.

\subsection{Superspaces}
The group algebra of the group $\mathz_2=\{\eve,\odd\}$
of order two over $\Bbbk$ is denoted by $\Bbbk\mathz_2$
which has a unique Hopf algebra structure over $\Bbbk$.
We let $\Cat:=(\Mod^{\Bbbk\mathz_2},\tens,\Bbbk)$ denote 
the tensor category of right $\Bbbk\mathz_2$-comodules,
that is, $\mathz_2$-graded modules over $\Bbbk$.
Here, $\Bbbk$ is regarded as a purely even object; $\Bbbk=\Bbbk\oplus0$.
In what follows, for $V=V_\eve\oplus V_\odd\in\Cat$,
an element $v$ in $V$ is always regarded as a homogeneous element of $V$.
We denote the {\it parity} of $v$ by $|v|$,
that is, $|v|=\epsilon$ if $v\in V_\epsilon$ ($\epsilon\in\mathz_2$).
For each $V,W\in\Cat$, the following is the so-called {\it supersymmetry}.
\[
c=c_{V,W} : V\tens W \longrightarrow W\tens V ;\quad
v\tens w \longmapsto (-1)^{|v||w|} w\tens v.
\]
In this way, $(\Mod^{\Bbbk\mathz_2},\tens,\Bbbk,c)$ forms 
a symmetric tensor category
which we denote by $\SMod$.
An object in $\SMod$ is called a {\it superspace}.
We let $\SMod(V,W)$ denote the set of all morphisms 
from $V$ to $W$ in $\SMod$.

For $V\in\SMod$, we define an object $\Pi V$ of $\SMod$
so that $(\Pi V)_\epsilon = V_{\epsilon+\odd}$ for each $\epsilon\in\mathz_2$.
For a morphism $f:V\to W$ in $\SMod$,
we define a morphism $\Pi f: \Pi V\to \Pi W$ in $\SMod$ so that $\Pi f=f$.
In this way, we get a functor $\Pi:\SMod \to \SMod$,
called the {\it parity change functor} on $\SMod$.
For $V,W\in\SMod$, we define an object
$\underline{\SMod}(V,W)$ of $\SMod$
as follows.
\[
\underline{\SMod}(V,W)_\eve := \SMod(V,W),
\quad
\underline{\SMod}(V,W)_\odd := \SMod(\Pi V,W).
\]
Given an object $V$ of $\SMod$,
we define a superspace $V^* := \underline{\SMod}(V,\Bbbk)$,
called the {\it dual} of $V$.

\medskip

A {\it superalgebra} (resp. {\it supercoalgebra/Hopf superalgebra/Lie superalgebra})
is defined to be an algebra (resp. coalgebra/Hopf algebra/Lie algebra) object in $\SMod$.
A superalgebra $A$ is said to be commutative
if $ab=(-1)^{|a||b|}ba$ for all $a,b\in A$.
For a supercoalgebra $C$ with the comultiplication $\Deltaa_C : C\to C\tens C$,
we use the Hyneman-Sweedler notation
$\Deltaa_C(c) = \sum_c c_{(1)} \tens c_{(2)}$,
where $c\in C$.

For a superalgebra $A$, we let ${}_A\SMod$ denote 
the category of all left $A$-supermodules.
Similarly, we let $\SMod^C$ denote
the category of all right $C$-supercomodules for 
a supercoalgebra $C$.
We also define a superspace 
$\underline{\SMod}^C(V,W) := \SMod^C(V,W) \oplus \SMod^C(\Pi V,W)$
for $V,W\in \SMod^C$.
The definition of $\underline{\SMod}_A(M,N)$ for $M,N\in\SMod_A$ will be clear.

\subsection{Algebraic Supergroups}\label{sec:alg}

An {\it affine supergroup scheme} ({\it supergroup}, for short) over $\Bbbk$
is a representable functor $\G$ from 
the category of commutative superalgebras
to the category of groups.
We denote the representing object of $\G$ by $\mathcal{O}(\G)$
which forms a commutative Hopf superalgebra.
A supergroup $\G$ is said to be {\it algebraic} (resp. {\it flat}) if
$A:=\mathcal{O}(\G)$ is finitely generated as a superalgebra (resp. $\Bbbk$-flat).
A {\it closed super-subgroup} of a supergroup $\G$ is a supergroup
which is represented by a quotient Hopf superalgebra of $A$.
For a supergroup $\G$,
we define its {\it even part} $\Gev$ as the restricted functor of $\G$
form the category of commutative algebras to the category of groups.
This $\Gev$ is an ordinary affine group scheme represented by
the quotient (ordinary) Hopf algebra $\overline{A}:=A/I_A$,
where $I_A$ is the super-ideal of $A$ generated by 
the odd part $A_\odd$ of $A$.
We denote the quotient map by
$A\twoheadrightarrow \overline{A}$; $a \mapsto \overline{a}$.
If $\G$ is algebraic, then so is $\Gev$.
An algebraic supergroup is said to be {\it connected}
if its even part is connected,
see \cite[Definition~8]{Mas12}.

\begin{ex}
For a superspace $V$ over $\Bbbk$,
we define a group functor $\GL(V)$ so that
\[
\GL(V)(R) \,\, := \,\, \underline{\Aut}_R(V\tens R),
\]
where $R$ is a commutative superalgebra
and $\underline{\Aut}_R(V\tens R)$
is the set of all bijective maps in
$\underline{\SMod}_R(V\tens R,V\tens R)$.
Suppose that $V$ is finite free with 
$m=\rank(V_\eve)$ and $n=\rank(V_\odd)$.
Then $\GL(V)$ becomes an algebraic supergroup,
called a {\it general linear supergroup},
and is denoted by $\GL(m|n)$.
The even part $\GL(m|n)_\ev$ is isomorphic to $\GLL_m\times \GLL_n$.
For the explicit form of the Hopf algebra structure of 
$\mathcal{O}(\GL(m|n))$,
see \cite{BruKuj03,Zub06} for example.
\qed
\end{ex}

Let $\G$ be a supergroup with $A:=\mathcal{O}(\G)$.
A superspace $M$ is said to be a {\it left $\G$-supermodule}
if there is a group functor morphism from $\G$ to $\GL(M)$.
We let ${}_{\G}\SMod$ denote the tensor category of left $\G$-supermodules.
One easily sees that $M$ forms a right $A$-supercomodule.
Conversely, any right $A$-supercomodule can be regard as a left $\G$-supermodule.
In this way, we may identify ${}_{\G}\SMod$ and $\SMod^{A}$.
Let $\H$ be a closed super-subgroup of $\G$ with $B:=\mathcal{O}(\H)$.
For a left $\G$-supermodule $M$ and a left $\H$-supermodule $V$,
we may consider its restricted $\H$-supermodule and induced $\G$-supermodule,
respectively;
\[
\res^{\G}_{\H}(M)\,\,:=\,\,\res^{A}_{B}(M)
\quad \text{and} \quad
\ind^{\G}_{\H}(V)\,\,:=\,\,\ind^{A}_{B}(V),
\]
via the quotient Hopf superalgebra map
$A\twoheadrightarrow B$,
see Appendix~\ref{app:cotens}.
Suppose that $\G$ and $\H$ are flat.
We also obtain the following Frobenius reciprocity, see \eqref{eq:FrobApp}.
\begin{equation} \label{eq:Frob}
\underline{\SMod}^B(\res^A_B(M),V) 
\,\,\overset{\cong}{\longrightarrow} \,\,
\underline{\SMod}^A(M,\ind^A_B(V)).
\end{equation}
Zubkov \cite[Propositon~3.1]{Zub06}
showed that the category ${}_{\H}\SMod$ of 
all left $\H$-supermodules has enough injectives.
Since the induction functor 
$\ind^{\G}_{\H}(-)$ is left exact,
we get its right derived functor 
$R^n \ind^{\G}_{\H}(-)$
form ${}_{\H}\SMod$ to ${}_{\G}\SMod$ ($n\in\mathz_{\geq0} = \{0,1,2,\dots\}$).

\medskip

A non-zero left $\G$-supermodule $L$
is said to be {\it simple} (or {\it irreducible}) 
if $L$ is simple as a right $A$-supercomodule,
that is, $L$ has no non-trivial $A$-super-subcomodule.
If $L$ is simple, then so is $\Pi L$.
As in \cite{BruKle03}, we shall use the following terminology.

\begin{deff} \label{def:MQ}
A simple left $\G$-supermodule $L$ is said to be of {\it type $\QQ$}
if $L$ is isomorphic to $\Pi L$ as left $\G$-supermodules,
and {\it type $\MM$} otherwise.
\end{deff}

Let $\Simple(\G)$ denote the set of isomorphism classes of 
simple left $\G$-supermodules.
The functor $\Pi$ naturally acts on $\Simple(\G)$
as a permutation of order $2$.
Let $\Simple_\Pi(\G)$ denote the set of $\langle \Pi \rangle$-orbits in $\Simple(\G)$.
In other words, two elements $L,L'$ in $\Simple(\G)$
coincides in $\Simple_\Pi(\G)$ if and only if $L\cong L'$ or $\Pi L'$ 
as left $\G$-supermodules.

\subsection{Lie superalgebras and Super-hyperalgebras}

In the following, we let $\G$ be an algebraic supergroup with $A:=\mathcal{O}(\G)$.
Set $A^+:=\Ker(\varepsilon_\G)$,
where $\varepsilon_\G : A \to \Bbbk$ is the counit of $A$.
As a $\Bbbk$-super-submodule of $A^*$, we let
\[
\Lie(\G) \,\ := \,\, (A^+/(A^+)^2)^*.
\]
This naturally forms a Lie superalgebra (see \cite[Proposition~4.2]{MasShi17}),
which we call the {\it Lie superalgebra of} $\G$.
Note that, $\Lie(\G)$ is finite free and
$\Lie(\G)_\eve$ can be identified with the (ordinary) Lie algebra $\Lie(\Gev)$ of $\Gev$.

For any $n\geq 1$, we may regard $(A/(A^+)^n)^*$ as 
a $\Bbbk$-super-submodule of $A^*$ through the dual of the canonical quotient map
$A\twoheadrightarrow A/(A^+)^n$.
Set
\[
\hy(\G) \,\, := \,\, \bigcup_{n\geq1} (A/(A^+)^n)^*.
\]
This $\hy(\G)$ forms a super-subalgebra of $A^*$.
We call it the {\it super-hyperalgebra of $\G$}.
It is sometimes called the {\it super-distribution algebra $\Dist(\G)$ of $\G$}.
Suppose that $\G$ is {\it infinitesimally flat}, that is,
$A/(A^+)^n$ is finitely presented and flat as $\Bbbk$-module 
(or equivalently, finite free) for any $n\geq1$.
Then one sees that $\hy(\G)$ has a structure of a cocommutative Hopf superalgebra
such that 
the restriction
\begin{equation} \label{eq:Hopf-pairing}
\langle\,\,\,,\,\, \rangle \,\,: \,\,
\hy(\G) \times A \longrightarrow \Bbbk
\end{equation}
of the canonical pairing $A^*\times A\to \Bbbk$
is a Hopf pairing,
see \cite[Lemma~5.1]{MasShi17}.
One sees that
$\Lie(\G)$ coincides with
the set of all {\it primitive elements}
$P(\hy(\G)) := 
\{u\in \hy(\G) \mid \sum_u u_{(1)}\otimes u_{(2)} = u\tens 1 + 1\tens u\}$
in $\hy(\G)$.

\begin{rem}[See \text{\cite[Remark~2]{Mas12}}] \label{rem:hy=U}
Suppose that $\Bbbk$ is a field of characteristic zero.
Then it is known that 
$\hy(\G)$ coincides with the universal enveloping superalgebra 
$\UU(\Lie(\G))$
of the Lie superalgebra $\Lie(\G)$ of $\G$.
\qed
\end{rem}

For a left $\G$-supermodule $V$,
we regard $V$ as a left $\hy(\G)$-supermodule by letting
\[
u . v \,\,:=\,\, 
\sum_v (-1)^{|v_{(0)}| |u|}\,v_{(0)}\, \langle u, v_{(1)} \rangle,
\]
where $u\in\hy(\G)$, $v\in V$ and $v\mapsto \sum_v v_{(0)}\tens v_{(1)}$
is the right $A$-supercomodule structure of $V$.
Suppose that $V$ is finite free.
Then the dual superspace $V^*$ of $V$
forms a right $A$-supercomodule by using the antipode $\s:A\to A$ of $A$.
The induced left $\hy(\G)$-supermodule structure,
which we shall denote by $\rightharpoonup$,
satisfies the following equation.
\[
\langle u \hits f , v \rangle \,\,=\,\,
(-1)^{|u| |f|} \langle f , \s^*(u) . v \rangle,
\]
where $v\in V$, $f\in V^*$, $u\in\hy(\G)$
and $\s^*$ is the antipode of $\hy(\G)$.

\subsection{Sub-pairs} \label{sec:sub-pairs}

Let $\G$ be an algebraic supergroup.
It is known that the Lie superalgebra $\Lie(\G)$ of $\G$ is {\it 2-divisible},
that is, 
for any $v\in \Lie(\G)_\odd$, there exists $x\in \Lie(\G)_\eve$ such that
$[v,v]=2x$, see \cite[Proposition~4.2]{MasShi17}.
We sometimes denote $x$ by $\frac{1}{2}[v,v]$.
Set $A:=\mathcal{O}(\G)$ and
\[
W^A \,\, := \,\, (A^+/(A^+)^2)_\odd.
\]
It is easy to see that the right coaction on $A$
\begin{equation} \label{eq:coad}
\coad : A \longrightarrow A\tens A;
\quad
a \longmapsto 
\sum_a (-1)^{|a_{(1)}||a_{(2)}|} a_{(2)} \tens \s(a_{(1)}) a_{(3)}
\end{equation}
induces on $W^A$ a right $\overline{A}$-comodule 
(or equivalently, left $\Gev$-module) structure,
where $\s$ is the antipode of $A$.
Thus, if $W^A$ is finite free,
then $\Lie(\G)_\odd$ can be regarded 
as a left $\Gev$-module.

Suppose that $\G$ satisfies the following three conditions.
\begin{enumerate}
\item[(E0)] $\Gev$ is flat,
\item[(E1)] $W^A$ is finite free, and
\item[(E2)] $\overline{A}^+/(\overline{A}^+)^2$ is finite free.
\end{enumerate}
The labels (E0)--(E2) above are taken from \cite[\S5.2]{MasShi18}.
In other words, we suppose that $\G$ is an object of 
the category $\text{(gss-fsgroups)}'_\Bbbk$,
see \cite[Remark~5.8]{MasShi18}
for the notation.
In this case, $\G$ is {\it $\tens$-split}, in the sense that
there exists a counit-preserving isomorphism
\begin{equation}
A \longrightarrow \overline{A} \tens \wedge (W^A)
\end{equation}
of left $\overline{A}$-comodule superalgebras,
where $\wedge(W^A)$ is the exterior superalgebra on $W^A$.
Here, we naturally regard $A$ as a left $\overline{A}$-comodule superalgebra
via $A\to \overline{A}\tens A$; 
$a\mapsto \sum_a \overline{a_{(1)}} \tens a_{(2)}$.

Note that, a $\Bbbk$-(super-)submodule of a finite free $\Bbbk$-(super)module
is again finite free, since $\Bbbk$ is a PID.
The following is a generalization of \cite[Definition~6.3]{MasShi18}.

\begin{deff}\label{def:sub-pair} 
Let $\G$ be an algebraic supergroup.
For a closed subgroup $K$ of $\Gev$ and 
a Lie super-subalgebra $\mathfrak{k}$ of $\Lie(\G)$
with $\Lie(K)=\mathfrak{k}_\eve$,
we say that the pair $(K,\mathfrak{k})$ is a {\it sub-pair} of
the pair $(\Gev,\Lie(\G))$, if
\begin{enumerate}
\item[(S0)] $K$ is flat,
\item[(S1)] $\mathfrak{k}_\odd$ is $K$-stable in $\Lie(\G)_\odd$, and
\item[(S2)] $\mathfrak{k}$ is 2-divisible.
\end{enumerate}
\end{deff}

\begin{rem} \label{rem:2-div}
In general, 
a Lie super-subalgebra $\mathfrak{k}$ of a $2$-divisible Lie superalgebra $\g$
is not always $2$-divisible.
For example, as super-subspaces of $\Mat_{1|1}(\mathz)$ 
(see Example~\ref{ex:Mat} for the notation), 
we take
\[
\mathfrak{k}=
\{ \left(\begin{array}{c|c} 2a & 3b \\ \hline 3b & 2a \end{array}\right)
\mid a,b,c\in\mathz \}
\quad \subseteq \quad
\g=
\{ \left(\begin{array}{c|c} a & b \\ \hline c & a \end{array}\right)
\mid a,b,c\in\mathz \}.
\]
The following calculation shows that this $\mathfrak{k}$ is {\it not} $2$-divisible.
\[
[\left(\begin{array}{c|c} 0 & 3 \\ \hline 3 & 0 \end{array}\right),
\left(\begin{array}{c|c} 0 & 3 \\ \hline 3 & 0 \end{array}\right)]
\,\,=\,\
2\left(\begin{array}{c|c} 9 & 0 \\ \hline 0 & 9 \end{array}\right).
\]
Working over a PID,
it is easy to see that
a Lie super-subalgebra of an {\it admissible} Lie superalgebra
(for the definition, see \cite[Definition~3.1]{MasShi17})
is admissible if and only is it is $2$-divisible.
\qed
\end{rem}

\begin{prop}\label{prp:sub-pair} 
Let $\G$ be an algebraic supergroup satisfying (E0)--(E2).
For a sub-pair $(K,\mathfrak{k})$ of the pair $(\Gev,\Lie(\G))$,
there exists a closed subgroup $\K$ of $\G$
which satisfies $\Kev=K$ and $\Lie(\K)=\mathfrak{k}$.
Moreover, $\K$ is $\tens$-split.
\end{prop}

\begin{proof}
We freely use the notations in \cite[\S5]{MasShi18}.
To construct the desired supergroup $\K$,
it is enough to show that the pair $(K,\mathfrak{k})$
is object of the category $\text{(sHCP)}'_\Bbbk
\,(\approx \text{(gss-fsgroups)}'_\Bbbk)$,
that is, $K$ and $\mathfrak{k}$ satisfy
the conditions (F0)--(F5) defined in \cite[\S5.1, 5.2]{MasShi18}.

The condition (F0) is just our (S0).
By the assumption $\Lie(K)=\mathfrak{k}_\eve$, the condition (F1) is obvious.
Since $\mathcal{O}(K)$ is a quotient Hopf algebra of $\mathcal{O}(\Gev)$ and
$\Gev$ satisfies the condition (E2),
one sees that $\mathcal{O}(K)$ satisfies the condition (F2).
Since $\G\in \text{(gss-fsgroups)}'_\Bbbk$,
the pair $(\Gev,\Lie(\G))$ is an object of the category $\text{(sHCP)}'_\Bbbk$.
Then by the assumption (S1), 
we see that $(K,\mathfrak{k})$ satisfies the conditions (F3) and (F4).
By the assumption (S2),
we can define a $2$-operation 
$(-)^{\langle2\rangle}:\mathfrak{k}_\odd \to\mathfrak{k}_\eve$
by letting $v^{\langle2\rangle}:=\frac12 [v,v]$.
This is obviously $K$-stable, which shows that the condition (F5) is satisfied.
\end{proof}

\section{Structures of Quasireductive Supergroups over a PID} \label{sec:Quasired}

It is known that
the category of all representations of a split reductive groups over a field
of characteristic zero
is semisimple,
that is,
any object is decomposed into a direct sum of simple modules 
(i.e., irreducible representations).
However,
algebraic supergroups $\G$
whose representation category is semisimple are rather restricted.
Over an algebraically closed field of characteristic zero,
Weissauer showed that 
such $\G$ (which are not ordinary algebraic groups)
are essentially exhausted by 
ortho-symplectic supergroups $\SpO(2n|1)$, see \cite[Theorem~6]{Wei09}.
On the other hand, over a field of positive characteristic,
Masuoka showed that
such $\G$ are necessarily purely even, that is, $\G=\Gev$, see \cite[Theorem~45]{Mas12}.

Hence, it is natural to ask 
what is a nice/good generalization of the notion of 
split reductive groups.
In 2011, Serganova \cite{Ser11} answered this problem.
She introduced the notion of ``quasireductive supergroups'' over a field.

In this section, we still suppose that $\Bbbk$
is a PID of characteristic not equal to $2$,
and we generalize her definition to our $\Bbbk$ and study its structure.

\subsection{Quasireductive Supergroups}

A split and connected reductive $\mathz$-group $G_\mathz$
is a certain connected algebraic group over $\mathz$
having a split maximal torus $T_\mathz$
such that the pair $(G_\mathz, T_\mathz)$ corresponds to 
some root datum,
see \cite[Part~II, Chap.~1]{Jan03} (see also \cite[\S5.2]{MasShi17}).
It is known that $\mathcal{O}(G_\mathz)$ is free as a $\mathz$-module and 
$G_\mathz$ is infinitesimally flat.

\begin{deff}\label{def:Quasired}
An algebraic supergroup $\G_{\mathz}$ over $\mathz$
is said to be {\it quasireductive} if
its even part $(\G_{\mathz})_{\ev}$ is 
a split and connected reductive group over $\mathz$
and $W^{\mathcal{O}(\G_\mathz)}$ is finitely generated and free as a $\mathz$-module.
\end{deff}

There are plenty of examples of quasireductive supergroups.

\begin{exs}
The following are quasireductive supergroups over $\mathz$.
\begin{enumerate}
\item
General linear supergroups $\GL(m|n)$.

\item
{\it Queer supergroups} $\Q(n)$:
\[
\Q(n)(R)
:=
\{ \left(\begin{array}{c|c} X & Y \\ \hline -Y & X \end{array}\right) 
\in \GL(n|n)(R)\},
\]
where $R$ is a commutative superalgebra over $\mathz$.
The even part is $\Q(n)_\ev=\GL_n$.
For the Hopf superalgebra structures of $\Q(n)$, 
see Brundan and Kleshchev \cite{BruKle03}.
The Lie superalgebra $\q(n)$ of $\Q(n)$ is 
the so-called {\it queer Lie superalgebra}.

\item
Chevalley supergroups of classical type 
defined by Fioresi and Gavarini \cite{FioGav12,FioGav13},
see also \cite[\S6]{MasShi17}.
For example, 
special linear supergroups $\SL(m|n)$ (consists of elements in $\GL(m|n)$
whose {\it Berezinian determinants} are trivial), 
ortho-symplectic supergroups $\SpO(m|n)$ 
(see Shu and Wang \cite{ShuWan08} for details), etc.

\item
{\it Periplectic supergroups} $\mathbf{P}(n)$, 
see Example~\ref{ex:P(n)}.
\qed
\end{enumerate}
\end{exs}

In what follows,
we fix a quasireductive algebraic supergroup $\G_\mathz$ over $\mathz$
and a split maximal torus $T_\mathz$ of $(\G_\mathz)_{\ev}$.
Set $G_\mathz:=(\G_{\mathz})_{\ev}$ and $\Lambda := \X(T_\mathz)$.
Note that, the character group $\X(T_\mathz)$ of $T_\mathz$
can be identified with the set of all {\it group-like elements} of 
the Hopf algebra $\mathcal{O}(T_\mathz)$.
Let $\G$ (resp. $G$, $T$) denote the base change of 
$\G_{\mathz}$ (resp. $G_\mathz$, $T_\mathz$) to our base ring $\Bbbk$,
that is, $\mathcal{O}(\G):=\mathcal{O}(\G_\mathz)\tens_\mathz \Bbbk$.
Since $\Bbbk$ is an integral domain,
we can identify $\Lambda$ with $\X(T)\cong \mathz^{\ell}$,
where $\ell$ is the rank of $G$.

By definition,
we see that the algebraic supergroup $\G$
satisfies the conditions (E0)--(E2).
Therefore, $\G$ is $\tens$-split:
\begin{equation} \label{eq:psi}
\psi\,:\,
\mathcal{O}(\G)
\,\, \overset{\cong}{\longrightarrow} \,\,
\mathcal{O}(G) \otimes \wedge (W^{\mathcal{O}(\G)}).
\end{equation}
In particular, we have the following.

\begin{prop} \label{prp:inf-flat}
$\mathcal{O}(\G)$ is $\Bbbk$-free and $\G$ is infinitesimally flat.
\end{prop}

Let $\g=\g_\eve\oplus\g_\odd$ be the Lie superalgebra $\Lie(\G)$ of $\G$.
The left adjoint action of $T$ on $\G$ induces an action of $T$ on $\g$
which preserves the parity of $\g$.
Since $T$ is a diagonalizable group scheme,
the left $T$-supermodule $\g$ decomposes into weight superspaces as follows.
\[
\g
\,\,=\,\,
\bigoplus_{\alpha\in \Lambda} \g^\alpha
\,\,=\,\,
(\bigoplus_{\alpha\in\Lambda} \g_{\eve}^\alpha )
\,\oplus\,
(\bigoplus_{\delta\in\Lambda} \g_{\odd}^\delta ),
\]
where $\g^\alpha$ is the $\alpha$-weight super-subspace of $\g$.
It is known that
\[
\g^\alpha
\,\,=\,\,
\{ X\in \g \mid u \hits X = \langle u, \alpha \rangle X \text{ for all } u\in\hy(T) \},
\]
see \cite[Part~I, 7.14]{Jan03}.
Here, we regard $\Lambda$ as a subset of $\mathcal{O}(T)$.
Let $\h:=\g^{\alpha=0}$ be the $\alpha=0$ weight super-subspace of $\g$
which is a Lie super-subalgebra of $\g$.
Note that, the even part $\g_\eve$ of $\g$ coincides with
the Lie algebra $\Lie(G)$ of $G=\Gev$.
By definition, we see that $\h_\eve = \Lie(T)$.
For $\epsilon\in\mathz_2$, we set
\[
\Delta_\epsilon
\,\,:=\,\,
\{ \alpha\in\Lambda \mid \g_\epsilon^{\alpha} \not=0 \} \setminus \{0\}
\]
and
\[
\Delta
\,\,:=\,\,
\begin{cases}
\Delta_\eve \cup \Delta_\odd            & \text{if } \h_{\odd}=0, \\
\Delta_\eve \cup \Delta_\odd \cup \{0\} & \text{otherwise.} 
\end{cases}
\]
We shall call $\Delta$ the {\it root system} of $\G$ with respect to $T$.
Note that, $\Delta_\eve$ is the root system of $G$ with respect to $T$ in the usual sense.
Suppose that $\G=\Q(n)$ with the standard maximal torus $T$ of $\Gev=\GL_n$.
Then one sees that $\h_\odd \not=0$, hence $0\in\Delta$.

\medskip

By taking the linear dual of \eqref{eq:psi}, 
we see that $\hy(\G)$ is {\it $\tens$-split},
that is,
there exists a (unique) unit-preserving isomorphism
\begin{equation} \label{eq:phi}
\phi \,:\,
\hy(G) \otimes \wedge(\g_\odd)
\,\, \overset{\cong}{\longrightarrow} \,\,
\hy(\G)
\end{equation}
of left $\hy(G)$-module supercoalgebras satisfying
\begin{equation} \label{eq:canopair}
\langle \phi(z),\, a \rangle
\,\,=\,\,
\langle z,\, \psi(a) \rangle,
\qquad 
a\in \mathcal{O}(\G),\,\, z\in \hy(G) \otimes \wedge(\g_\odd).
\end{equation}
For details, see \cite[Lemma~5.2]{MasShi17}.

The dual $\hy(\G)^*$ of the cocommutative Hopf superalgebra $\hy(\G)$ 
naturally has a structure of a commutative superalgebra.
Since our quasireductive supergroup $\G$ is connected,
the canonical Hopf paring \eqref{eq:Hopf-pairing} induces
an injection $\mathcal{O}(\G)\hookrightarrow \hy(\G)^*$ of superalgebras,
see \cite[Lemma~5.3]{MasShi17}.

A left $\hy(\G)$-supermodule $M$ is called a 
{\it left $\hy(\G)$-$T$-supermodule}
if the restricted left $\hy(T)$-module structure on $M$
arises from some left $T$-module structure on $M$.
In this case, $M$ is said to be {\it locally finite}
if $M$ is so as a left $\hy(\G)$-supermodule.
Then by \cite[Theorem~5.8]{MasShi17},
we have the following.

\begin{thm} \label{prp:MS5.8}
For a left $\G$-supermodule,
the induced left $\hy(\G)$-supermodule is
a locally finite left $\hy(\G)$-$T$-supermodule.
This gives an equivalence of categories between
$\Bbbk$-free, left $\G$-supermodules
and
$\Bbbk$-free, locally finite left $\hy(\G)$-$T$-supermodules.
\end{thm}

\subsection{Super-tori and Unipotent Super-subgroups} \label{sec:tori,unip}

Recall that, $G=\Gev$, $\Lie(T)=\h_\eve$ and $\mathcal{O}(T)$ is a free $\mathz$-module.

\begin{lemm} \label{prp:tori-pair}
The pair $(T,\h)$ forms a sub-pair of $(G,\g)$.
\end{lemm}

\begin{proof}
The condition (S0) is obvious.
By definition, we see that $\h$ is $\hy(T)$-stable,
and hence (S1) is clear.
To see (S2), we take $K\in\h_\odd$ and $X\in\g$ so that $[K,K]=2X$.
For all $u\in\hy(T)$ with $u\not\in\Bbbk$,
we see that $2(u\hits X)=u\hits [K,K]=0$.
Since $\charr(\Bbbk)\not=2$, we conclude that $X\in\h_\eve$.
\end{proof}

Then by Proposition~\ref{prp:sub-pair},
we get a closed super-subgroup $\T$ of $\G$ satisfying 
\begin{equation} \label{eq:T}
\Tev = T,\qquad \Lie(\T) = \h.
\end{equation}
We shall call $\T$ a {\it super-torus of $\G$}.

Let $\mathz\Delta$ be the abelian subgroup of $\X(T)$ generated by $\Delta$.
In what follows,
we fix a homomorphism $\gamma : \mathz\Delta \to \mathr$ of abelian groups
such that $\gamma(\alpha)\not=0$ for any $0\not=\alpha\in\Delta$.
As in \cite[\S9.3]{Ser11}, we set
\[
\Delta^\pm
\,\,:=\,\,
\{ \alpha\in\Delta \setminus \{ 0 \} \mid \pm\gamma(\alpha)>0 \},
\qquad
\Delta_\epsilon^\pm 
\,\,:=\,\,
\Delta_\epsilon \cap \Delta^\pm
\]
for $\epsilon\in\mathz_2$.
Set
\[
\n^\pm
\,\,:=\,\,
\bigoplus_{\alpha\in\Delta^\pm} \g^\alpha,
\qquad
\bb^\pm
\,\,:=\,\,
\n^\pm \oplus \h.
\]
One easily sees that, these are Lie super-subalgebras of $\g$.

As in \cite[Part~II, 1.8]{Jan03}, 
there is a connected and unipotent subgroup $U$ of $G$
such that
\[
U
\,\, \cong \,\,
\prod_{\alpha\in\Delta_{\eve}^-} \Ga,
\qquad
\Lie(U)
\,\,=\,\,
\n_{\eve}^-,
\]
where $\Ga$ is the one-dimensional additive group scheme over $\Bbbk$.
In the product above, we have fixed an (arbitrary) ordering of $\Delta_{\eve}^-$.

\begin{lemm} \label{prp:unip-pair}
The pair $(U, \n^-)$ forms a sub-pair of $(G,\g)$.
\end{lemm}

\begin{proof}
The condition (S0) is clear.
Take $X\in\g^\delta$ ($\delta\in\Delta^-_\odd$)
and $Y\in\g$ so that $[X,X]=2Y$.
For all $u\in\hy(T)$,
we see that $2(u\hits Y) = 2\langle u, \delta \rangle [X,X]$.
Since $2\delta\in\Delta^-_\eve$ and $\charr(\Bbbk)\not=2$,
we conclude that $Y\in \n^-_\eve$, and hence $\n^-$ is $2$-divisible (S2).
It remains to prove (S1), that is, 
$\n_{\odd}^-$ is $U$-stable in $\g_\odd$.
Since $\n_{\odd}^-$ is obviously $G$-stable in $\g_\odd$,
we see that it is a right $\hy(G)$-module via 
the induced action.

First, we show that $\n_{\odd}^-$ is $\hy(U)$-stable in $\g_\odd$.
It is known that
we can choose an element $X_\alpha\in\g_{\eve}^\alpha$ for each 
$\alpha\in\Delta_{\eve}^-$ so that
the set
\[
\{ \prod_{\alpha\in\Delta_{\eve}^-} X_\alpha^{(n_\alpha)} \mid 
n_\alpha=0,1,2,3,\dots \}
\]
forms a $\Bbbk$-basis of $\hy(U)$.
Here,
$X_\alpha^{(n_\alpha)}$
denotes the {\it divided power},
see \cite[Part~II, 1.12]{Jan03} (see also \S\ref{sec:Kostant} below).
We take $X_{\beta}\in \n_{\odd}^-$ and $X_{\alpha}^{(n_\alpha)} \in \hy(U)$.
It is easy to see that 
the $T$-weight of $X_\alpha^{(n_\alpha)} \hits X_\beta \,(\in\g)$
is given by $\beta+n_\alpha\alpha$.
Since $\beta\in\Delta_{\odd}^-$ and $\alpha\in\Delta_{\eve}^-$,
we see $\gamma(\beta+n_\alpha\alpha) = \gamma(\beta)+n_\alpha\gamma(\alpha)\leq 0$,
and hence $\beta+n_\alpha\alpha$ is in $\Delta_{\odd}^-$.
Thus, $X_\alpha^{(n_\alpha)} \hits X_\beta$ is an element of $\n_{\odd}^-$.
This implies that $\n_{\odd}^-$ is $\hy(U)$-stable in $\g_\odd$.

Next, we show that $\n_{\odd}^-$ is indeed $U$-stable.
By construction, 
the corresponding Hopf algebra $\mathcal{O}(U)$ of $U$ is
isomorphic to the polynomial algebra $\Bbbk[T_\alpha;\,\alpha\in\Delta_{\eve}^-]$
in $\#(\Delta_{\eve}^-)$-variables $T_\alpha$ over $\Bbbk$ with each $T_\alpha$ is primitive.
Since $\n_{\odd}^-$ is finite free,
the right $\hy(U)$-module structure on $\n_{\odd}^-$ induces
the following (well-defined) left $\mathcal{O}(U)$-comodule structure on $\n_{\odd}^-$.
\[
\n_{\odd}^- \longrightarrow \mathcal{O}(U) \otimes \n_{\odd}^-;
\quad
Y \longmapsto 
\sum_{\alpha\in\Delta_{\eve}^-} \sum_{n_\alpha \geq 0}
T_\alpha^{n_\alpha} \otimes (X_\alpha^{(n_\alpha)} \hits Y),
\]
see \cite[Part~II, 1.20]{Jan03} for detail.
Thus, we conclude that $\n_{\odd}^-$ is $U$-stable.
\end{proof}

By Proposition~\ref{prp:sub-pair},
we obtain a closed super-subgroup $\U$ of $\G$
which satisfies
\begin{equation} \label{eq:U}
\Uev \,\,=\,\, U,
\qquad
\Lie(\U) \,\,=\,\, \n^-.
\end{equation}
The supergroup $\U$ is connected,
and hence the canonical Hopf paring induces
an injection $\mathcal{O}(\U)\hookrightarrow \hy(\U)^*$
of superalgebras, as before.
By \cite[Theorem~41]{Mas12}, we get the following.

\begin{prop} \label{prp:unip}
Suppose that $\Bbbk$ is a field.
Then the supergroup $\U$ is unipotent,
that is, the corresponding Hopf superalgebra $\mathcal{O}(\U)$ is irreducible.
\end{prop}

Similarly, we get
a closed super-subgroup $\U^+$ of $\G$
such that $\Uevp = U^+$, $\Lie(\U^+)=\n^+$,
where $U^+$ is an unipotent subgroup of $G$ corresponds to $\Delta^+_\eve$.

\subsection{Borel Super-subgroups} \label{sec:Borel}
Our next aim is to construct ``Borel super-subgroups'' of $\G$.

\begin{lemm} \label{prp:T-ad-U}
$\T$ normalizes $\U$.
\end{lemm}

\begin{proof}
First, we work over $Q:=\Frac(\Bbbk)$
and show that
$\T_{Q} := \T \times_{\Bbbk} Q$ is included in
the {\it normalizer} 
$\mathcal{N}_{\G_{Q}}(\U_{Q})$
of $\U_Q$ in $\G_Q$.
In the following, for simplicity,
we shall drop the subscript $Q$.
By \cite[Thorem~6.6(1)]{MasShi18},
to show $\T \subseteq \mathcal{N}_{\G}(\U)$,
it is enough to check the following conditions (i) and (ii):
\begin{enumerate}
\item[(i)]
$T$ is contained in $\mathcal{N}_G(U) \cap \Stab_G(\n_\odd^-)$;
\item[(ii)]
$\h_{\odd}$ is contained in
$\Inv_U(\g_\odd / \n_{\odd}^-) \cap (\n_{\eve}^- : \n_{\odd}^-)$.
\end{enumerate}
Here, $\Stab_G(\n_1^-)$ is the {\it stabilizer} of $\n_\odd^-$ in $G$,
$\Inv_U(\g_\odd / \n_{\odd}^-)$ is the largest $U$-submodule of $\g_\odd$
including $\n_\odd^-$ whose quotient $U$-module by $\n_\odd^-$ is trivial,
and
$(\n_{\eve}^- : \n_{\odd}^-) := 
\{ X\in \g_\odd \mid [X,\n_\odd^-] \subseteq \n_\eve^-\}$.

By the definition of root spaces, $\n_{\odd}^-$ is trivially $T$-stable.
Since $T\subseteq \mathcal{N}_G(U)$ is obvious,
the condition (i) is clear.
It is easy to see that $[\h_\odd,\n_\odd^-]\subseteq \n_{\eve}^-$.
Thus, we have $\h_\odd\subseteq (\n_{\eve}^-:\n_{\odd}^-)$.
By using the Hopf pairing $\hy(U)\otimes \mathcal{O}(U)\to Q$,
one sees that 
\[
\Inv_U(\g_\odd /\n_{\odd}^-)
\,=\,
\{ X\in\g_\odd \mid (u \hits X) - \varepsilon_U^*(u)X 
\in\n_{\odd}^- \text{ for all } u
\in\hy(U) \},
\]
where $\varepsilon_U^*$ is the counit of $\hy(U)$.
For each $X\in \h_\odd$ and $u = X_\alpha^{(n)}\in\hy(U)$
($\alpha\in\Delta_{\eve}^-$, $n\in\mathz_{\geq1}$),
the weight of $u \hits X$ is $n\alpha\in\Delta^-$.
Since $\varepsilon_U^*(u)=0$, we have
$(u\hits X) - \varepsilon_U^*(u)X \in \n_{\odd}^-$.
If $u\in Q$, then $(u \hits X) - \varepsilon_U^*(u)X=0$.
This shows $\h_\odd \subset \Inv_U(\g_\odd/\n_{\odd}^-)$, 
and hence the condition (ii) follows.
Thus, we get $\T\subseteq \mathcal{N}_{\G}(\U)$ over $Q$.

Next, we prove that $\T\subseteq \mathcal{N}_{\G}(\U)$ over $\Bbbk$.
Since $\mathcal{O}(\T)$ is $\Bbbk$-free, 
the canonical map $\mathcal{O}(\T)\to \mathcal{O}(\T_{Q})$ is injective.
Then one easily sees that the canonical quotient map
$\mathcal{O}(\mathcal{N}_{\G_{Q}}(\U_{Q}))
\twoheadrightarrow \mathcal{O}(\T_{Q})$
induces the following diagonal arrow:
\[
\begin{xy}
(0,0)  *+{\mathcal{O}(\G)}  ="1",
(25,0)  *+{\mathcal{O}(\mathcal{N}_{\G}(\U))}  ="2",
(0,-15)  *+{\mathcal{O}(\T)}  ="3",
(7,-5) *{\circlearrowright},
{"1" \SelectTips{cm}{} \ar @{->>} "2"},
{"1" \SelectTips{cm}{} \ar @{->>} "3"},
{"3" \SelectTips{cm}{} \ar @{<--} "2"},
\end{xy}
\]
Here,
the horizontal and vertical arrows are the canonical quotient maps.
Thus, we are done.
\end{proof}

Let $\T \ltimes \U$ be the crossed product supergroup scheme of $\T$ and $\U$.
Namely, for any superalgebra $R$,
the group $(\T \ltimes \U)(R)$ is just $\T(R) \times \U(R)$ as a set,
and the group multiplication is given by
\[
(t, u) \, (s, v)
\,\,=\,\,
(ts,\,
s^{-1}usv),
\quad 
s,t \in \T(R),\,\,
u,v \in \U(R).
\]
We define a closed super-subgroup $\B$ of $\G$
as the {\it image} (in the sense of \cite[Part~I, 5.5]{Jan03}) of the multiplication map $\T\ltimes \U \to \G$.
Since $T\cap U$ and $\h\cap \n^-$ are trivial,
the intersection $\T\cap \U$ is also trivial.
Thus, the multiplication map of $\G$ induces an isomorphism
$\T \ltimes \U \cong \B$
of algebraic supergroups.
By dualizing,
we have an isomorphism of Hopf superalgebras
\begin{equation} \label{eq:Bsplit}
\mathcal{O}(\B)
\overset{\cong}{\longrightarrow}
\mathcal{O}(\T \ltimes \U) \, = \, 
\mathcal{O}(\T) \cmdblackltimes \mathcal{O}(\U),
\end{equation}
where $\mathcal{O}(\T) \cmdblackltimes \mathcal{O}(\U)$ denotes the {\it Hopf-crossed product} of 
$\mathcal{O}(\U)$ and $\mathcal{O}(\T)$
in $\SMod$.
Note that, 
$\mathcal{O}(\T) \cmdblackltimes \mathcal{O}(\U)$ is just $\mathcal{O}(\T) \otimes \mathcal{O}(\U)$ as a superalgebra.
The identification $\B=\T \ltimes \U$ implies that 
\begin{equation} \label{eq:B}
B := \Bev = T \ltimes U,
\qquad
\Lie(\B) = \bb^-.
\end{equation}
Thus, $B$ is a Borel subgroup of $G$.
In this sense, we shall call $\B$ a {\it Borel super-subgroup of $\G$}
({\it with respect to $\gamma$}).
The supergroup $\B$ is connected,
and hence the canonical Hopf paring induces
an injection $\mathcal{O}(\B)\hookrightarrow \hy(\B)^*$
of superalgebras, as before.
Similarly, we define $\B^+:=\mathrm{Im}(\T\ltimes \U^+ \to \G)\cong \T\ltimes \U^+$.

\begin{rem} \label{rem:gamma}
Note that, the groups $\U$, $\U^+$, $\B$ and $\B^+$ are depending on
the choice of the homomorphism $\gamma:\mathz\Delta\to \mathr$,
defined in \S\ref{sec:tori,unip}.
In addition,
not all (possible) Borel super-subgroups are conjugate.
\qed
\end{rem}

Since $T \subseteq B$ and $\h \subseteq \bb^-$,
we see that $\T$ is a closed super-subgroup of $\B$.
Then the inclusion map $\T \hookrightarrow \B$ induces
a Hopf superalgebra surjection $\mathcal{O}(\B) \twoheadrightarrow \mathcal{O}(\T)$.

\begin{prop} \label{prp:split}
The morphism $\mathcal{O}(\B) \twoheadrightarrow \mathcal{O}(\T)$ is split epic.
\end{prop}

\begin{proof}
We consider the following Hopf superalgebra map
\begin{equation}\label{eq:section}
\spl_{\B}: 
\mathcal{O}(\T) \overset{1\otimes \id}{\longrightarrow} 
\mathcal{O}(\T) \cmdblackltimes \mathcal{O}(\U)
\overset{\cong}{\longrightarrow}
\mathcal{O}(\B),
\end{equation}
where the first arrow is given by $x\mapsto 1\otimes x$ and
the second arrow is the inverse map of \eqref{eq:Bsplit}.
Then one easily sees that 
this $\spl_{\B}$ gives a section of $\mathcal{O}(\B) \twoheadrightarrow \mathcal{O}(\T)$.
\end{proof}

\subsection{Kostant's $\mathz$-Form and PBW Theorem} \label{sec:Kostant}
As a superspace,
we have
\[
\g \,\,= \,\,
(\h_\eve\oplus \bigoplus_{\alpha\in\Delta_\eve} \g_{\eve}^\alpha)
\,\oplus\,
(\h_\odd\oplus \bigoplus_{\delta\in\Delta_\odd} \g_{\odd}^\delta).
\]
Note that, for $\alpha\in\Delta_\eve$, the rank of $\g_\eve^\alpha$ is always $1$.
However, for $\delta\in\Delta_\odd$, 
the rank of $\g_\odd^\delta$ may be greater than $1$,
see \cite[Lemma~9.6]{Ser11}.
For $\epsilon\in\mathz_2$, we set
\[
\ell_{\epsilon} \,\,:=\,\,
\rank(\h_\epsilon).
\]

Set $\g_\mathz:=\Lie(\G_\mathz)$ and $\h_\mathz:=\h\cap \g_\mathz$.
For $X \in (\g_\mathz)_\eve \setminus (\h_\mathz)_\eve$ and $H\in(\h_\mathz)_\eve$,
as elements in $\hy(G_\mathz) \otimes_\mathz \mathq$, we set
\[
X^{(n)} := {X^n}\tens_\mathz \frac{1}{n!},
\qquad
H^{(m)}  := (\prod_{j=0}^{m-1}(H-j))\tens_\mathz\frac{1}{m!},
\]
where $n,m\in\mathz_{\geq0}$.
Here, we set $H^{(0)}:=1$.
As in \cite[Part~II, 1.11]{Jan03},
we choose a $\mathz$-free basis of $(\g_\mathz)_\eve$
\[
\mathcal{B}_\eve:=
\{ X_\alpha \in (\g_\mathz)_\eve^\alpha \mid \alpha\in\Delta_\eve \}
\cup
\{ H_i \in (\h_\mathz)_\eve \mid 1\leq i \leq \ell_\eve\}
\]
so that 
the set of all products of factors of type
$X_\alpha^{(n_\alpha)}$ and $H_i^{(m_i)}$
($n_\alpha,m_i\in\mathz_{\geq0}$, $\alpha\in\Delta_\eve$, $1\leq i\leq \ell_\eve$),
taken in $\hy(G_\mathq)$
with respect to any order on $\mathcal{B}_\eve$ forms 
a $\mathz$-basis of $\hy(G_\mathz)$.

Since $(\g_\mathz)_\odd$ is also $\mathz$-free,
we take a $\mathz$-free basis of $\g_\mathz$ as follows.
\[
\mathcal{B}_\odd:=\{ Y_{(\delta,p)} \in (\g_\mathz)_\odd^\delta \mid 
\delta\in\Delta_\odd,
1\leq p \leq \rank(\g_\odd^\delta) \}
\cup
\{ K_{t} \in (\h_\mathz)_\odd \mid 1\leq t \leq \ell_\odd \}.
\]
We see that $\hy(\G)$ can be identified with $\hy(\G_\mathz)\otimes_\mathz \Bbbk$,
since $\G_\mathz$ is infinitesimally flat (Proposition~\ref{prp:inf-flat}).
As elements of $\hy(\G) = \hy(\G_\mathz)\otimes_\mathz \Bbbk$,
we set
\begin{align*}
X_{\alpha,n} &:= X_{\alpha}^{(n)}\otimes_\mathz 1,
 &
H_{i,m} &:= H_i^{(m)}\otimes_\mathz 1,
\\
Y_{(\delta,p),\epsilon} &:= Y_{(\delta,p)}^{\epsilon}\otimes_\mathz 1,
 &
K_{t,\epsilon} &:= K_t^{\epsilon}\otimes_\mathz 1,
\end{align*}
where $\epsilon = 0$ or $1$.
Then the supercoalgebra isomorphism 
$\hy(\G) \cong \hy(\Gev)\tens \wedge (\g_\odd)$
given in \eqref{eq:phi}
implies the following Poincar\'e-Birkhoff-Witt (PBW) theorem for $\hy(\G)$.

\begin{thm} \label{prp:PBW}
For any total order on the set $\mathcal{B}_\eve\cup\mathcal{B}_\odd$,
the set of all products of factors of type
\[
H_{i,m(i)}, \quad
X_{\alpha,n(\alpha)}, \quad
K_{t,\epsilon(t)},\quad
Y_{(\delta,p),\epsilon(\delta,p)}
\]
($n(\alpha),m(i)\in\mathz_{\geq0}$, $\alpha\in\Delta_\eve$, $1\leq i\leq \ell_\eve$,
$\delta\in\Delta_\odd$, $1\leq p \leq \rank(\g_\odd^\delta)$,
$1\leq t \leq \ell_\odd$
and
$\epsilon(t), \epsilon(\delta,p) \in \{0,1\}$),
taken in $\hy(\G)$
with respect to the order, forms a $\Bbbk$-basis of $\hy(\G)$.
\end{thm}

\begin{rem}
To construct Chevalley supergroups over $\mathz$,
it is necessary to prove that
(1) any finite-dimensional simple Lie superalgebra $\mathfrak{s}$ over $\mathc$ 
has a {\it Chevalley basis};
(2) the {\it Kostant's $\mathz$-form} of 
$\UU(\mathfrak{s})$ has a PBW basis, like above.
These were done by Fioresi and Gavarini \cite[Theorems~3.7,~4.7]{FioGav12}.
\end{rem}

Since the $\Bbbk$-valued points of $\G$ is just $\G(\Bbbk)=\Gev(\Bbbk)$,
it is hard to consider 
a geometrical ``denseness''
of $\B\times\B^+$ in $\G$ (or, a ``big cell'' as in \cite[Part~II, 1.9]{Jan03})
directly.
The following is an algebraic interpretation of the denseness,
see also \cite{ParWan91} and \cite{BicRic16}.

\begin{cor} \label{prp:inj}
The following superalgebra map is injective.
\begin{equation}\label{eq:inj}
\mathcal{O}(\G) \longrightarrow 
\mathcal{O}(\G)\otimes \mathcal{O}(\G) \longrightarrow 
\mathcal{O}(\B)\otimes \mathcal{O}(\B^+),
\end{equation}
where the first arrow is the comultiplication map of $\mathcal{O}(\G)$ and 
the second arrow is the tensor product of the canonical quotient maps
$\mathcal{O}(\G) \twoheadrightarrow \mathcal{O}(\B)$ and $\mathcal{O}(\G) \twoheadrightarrow \mathcal{O}(\B^+)$.
\end{cor}

\begin{proof}
As in the proof Proposition~\ref{prp:split},
we also get an injection $\mathcal{O}(\U) \hookrightarrow \mathcal{O}(\B)$.
Thus, to prove the claim, it is enough to see that
the map
$\mathcal{O}(\G) \to \mathcal{O}(\U)\otimes \mathcal{O}(\B^+)$ is injective.

The multiplication map $\mu:\U\times \B^+ \to \G$ induces
a morphism $\hy(\mu):\hy(\U)\otimes \hy(\B^+) \to \hy(\G)$ of supercoalgebras
which is indeed bijective, by $\n^-\oplus \bb^+=\g$ and Theorem~\ref{prp:PBW}.
Then the $\Bbbk$-linear dual $\hy(\mu)^*$ is an isomorphism of superalgebras.
Since $\hy(\U)$ and $\hy(\B^+)$ are both $\Bbbk$-free,
the canonical map 
$\hy(\U)^*\otimes \hy(\B^+)^* \to \big(\hy(\U)\otimes \hy(\B^+)\big)^*$ is injective.
Therefore, we get the following commutative diagram of superalgebras:
\[
\begin{xy}
(-40,7)   *++{\hy(\G)^*}  ="1",
(0,7)  *++{\big(\hy(\U)\tens \hy(\B^+)\big)^*}    ="2",
(45,7)  *++{\hy(\U)^*\tens \hy(\B^+)^*}    ="22",
(-40,-7) *++{\mathcal{O}(\G)}          ="3",
(45,-7)*++{\mathcal{O}(\U)\tens \mathcal{O}(\B^+).}    ="4",
(0,0)*{\circlearrowleft},
{"1" \SelectTips{cm}{} \ar @{->}_{\cong 
\phantom{12345678}}^{\hy(\mu)^* \phantom{1234567}} "2"},
{"3" \SelectTips{cm}{} \ar @{->} "1"},
{"4" \SelectTips{cm}{} \ar @{->} "22"},
{"22" \SelectTips{cm}{} \ar @{_(->} "2"},
{"3" \SelectTips{cm}{} \ar @{->}_{} "4"}
\end{xy}
\]
The vertical canonical maps are injection, since $\G$, $\U$ and $\B^+$ are connected.
Therefore, the lower horizontal arrow is also injective.
\end{proof}

\section{Borel-Weil Theorem for $\G$} 
\label{sec:simples}

Throughout the rest of the paper,
we assume that $\Bbbk$ is a field of characteristic different from 2.
All supergroups $\G$, $\B$, $\U$, $\T$, etc. are defined over the field $\Bbbk$.
In this section, we will construct all simple $\G$-supermodules,
which extends Serganova's construction \cite[\S9]{Ser11} to arbitrary characteristic.
The main idea is based on Brundan and Kleshchev's argument \cite[\S6]{BruKle03},
see also Parshall and Wang \cite{ParWan91},
Bichon and Riche \cite{BicRic16}
for the non-super situation.

\subsection{Simple $\T$-Supermodules}\label{sec:irrT}
We will construct 
all simple supermodules of the super-torus $\T$ of $\G$.
In the following, we freely use the notations used in Appendix~\ref{sec:clif}.

Recall that, $\ell_\epsilon$ denotes the dimension of
$\h_\epsilon$ for $\epsilon\in\mathz_2$.
By Theorem~\ref{prp:PBW},
the super-hyperalgebra $\hy(\T)$ of $\T$ has a $\Bbbk$-basis
\begin{equation} \label{eq:hy(T)-basis}
\{
\prod_{i=1}^{\ell_\eve} H_{i,m(i)} \, 
\prod_{t=1}^{\ell_\odd} K_{t,\epsilon(t)}
\mid
m(i)\in \mathz_{\geq0},\,\,\epsilon(t) \in \{0,1\}
\}.
\end{equation}
In the following, we fix $\lambda\in \Lambda$
and write $\lambda(H):=\langle H, \lambda \rangle$ 
for $H\in \h_{\eve}\,(\subseteq \hy(T))$.
Since $H_i\in(\h_\mathz)_{\eve}$, we see that $\lambda(H_i)\in\mathz$.
We let $\hy(\T)^\lambda$ denote the 
quotient superalgebra of 
the super-hyperalgebra $\hy(\T)$ of $\T$
by the two-sided super-ideal of $\hy(\T)$ generated by all
$H_{i,m} - \binom{\lambda(H_i)}{m}$,
where $1\leq i\leq \ell_\eve$ and $m\geq 0$.
Here, we used
\[
\binom{c}{m}
\,\,:=\,\,
\frac{1}{m!}
\prod_{j=0}^{m-1}
(c-j),
\quad
\binom{c}{0}
\,\,:=\,\,
1
\]
for $m\geq1$ and $c\in\mathz$.
Then we get the following.
\[
\dim \hy(\T)^\lambda \,\, = \,\, 2^{\ell_\odd}.
\]

For $x,y\in\h_\odd$,
we see that $[x,y] \in \h_\eve$,
and hence 
we can define a bilinear map $b^\lambda : \h_\odd\times \h_\odd \to \Bbbk$ as follows.
\[
b^\lambda(x,y) \,\, := \,\, \lambda([x,y]),
\]
where $x,y\in\h_\odd$.
It is easy to see that $(\h_\odd,b^\lambda)$ forms a quadratic space over $\Bbbk$.
Thus, we get the Clifford superalgebra 
$\C(\h_\odd, b^\lambda) := T(\h_\odd) / I(\h_\odd,b^\lambda)$ over $\Bbbk$.
We may regard $\h_\odd$ as a subspace of $\hy(\T)$,
since $\h$ coincides with the set of all primitive elements in $\hy(\T)$.
Then we have the following map.
\begin{equation} \label{eq:Ttohy^}
T(\h_\odd) \longrightarrow \hy(\T) \longrightarrow \hy(\T)^\lambda,
\end{equation}
where $T(\h_\odd)$ is the tensor algebra of $\h_\odd$.

\begin{lemm} \label{prp:C=hy^}
The map above induces an isomorphism 
$\C(\h_\odd,b^\lambda) \cong \hy(\T)^\lambda$ of superalgebras.
\end{lemm}

\begin{proof}
It is easy to see that \eqref{eq:Ttohy^} is surjective
and the kernel contains the ideal $I(\h_\odd,b^\lambda)$ of $T(\h_\odd)$.
Thus, \eqref{eq:Ttohy^} induces a surjection 
$\C(\h_\odd,b^\lambda) \to \hy(\T)^\lambda$ of superalgebras.
On the other hand,
it is known that the dimension of $\C(\h_\odd,b^\lambda)$ is 
$2^{\ell_\odd}$.
Thus, we are done.
\end{proof}

Take a non-degenerate subspace $(\h_\odd)_s$ of $\h_\odd$
so that $\h_\odd = \rad(b^\lambda) \perp (\h_\odd)_s$,
and set 
\begin{equation}
d_{\lambda} \,\, := \,\, \ell_\odd - \dim \big(\rad(b^\lambda)\big).
\end{equation}
We choose an orthogonal basis $\{x_1,\dots, x_{d_\lambda}\}$ of $(\h_\odd)_s$
and we let 
\begin{equation}
\delta_{\lambda} \,\,:=\,\, 
(-1)^{d_{\lambda}(d_{\lambda}+1)/2} \,
\lambda([x_1,x_1]) \cdots \lambda([x_{d_\lambda},x_{d_\lambda}]),
\end{equation}
the signed determinant of $(\h_\odd)_s$,
see \eqref{eq:delta}.
For simplicity, we let $\delta_\lambda=0$ if $b^\lambda=0$.
By Lemma~\ref{prp:C=hy^} and Proposition~\ref{prp:Clif-simples},
we have the following.

\begin{prop} \label{prp:u(lambda)}
The superalgebra $\hy(\T)^\lambda$ has a unique simple supermodule 
$\mathfrak{u}(\lambda)$ up to isomorphism and parity change.
If (1) $\delta_\lambda=0$ or 
(2) $d_{\lambda}$ is even and $\delta_{\lambda} \in (\Bbbk^\times)^2$,
then $\Pi \mathfrak{u}(\lambda) \not= \mathfrak{u}(\lambda)$.
Otherwise, $\Pi \mathfrak{u}(\lambda) = \mathfrak{u}(\lambda)$.
\end{prop}

For a left $\hy(\T)^\lambda$-supermodule $M$,
we may regard $M$ as a left $\hy(\T)$-supermodule via the canonical quotient map
$\hy(\T)\to \hy(\T)^\lambda$.
It is easy to see that $M$ is a left $\hy(\T)$-$T$-supermodule.
Since $\uu(\lambda)$ is finite-dimensional,
it is obviously locally finite as a left $\hy(\T)$-$T$-supermodule.

\begin{lemm} \label{prp:ht-T}
Let $L$ be a locally finite left $\hy(\T)$-$T$-supermodule.
If $L$ is simple, then there exists $\lambda\in \Lambda$ such that 
$L$ is isomorphic to $\uu(\lambda)$ or $\Pi \uu(\lambda)$.
\end{lemm}

\begin{proof}
Since $L$ is non-zero,
there exists $\lambda\in \Lambda$ such that 
the $\lambda$-weight space $L^\lambda$ of $L$ is non-zero.
By definition, $L^\lambda$ is a $\hy(\T)^\lambda$-supermodule.
Hence by Proposition~\ref{prp:u(lambda)},
$L^\lambda$ contains $\uu(\lambda)$ or $\Pi \uu(\lambda)$, say $\uu(\lambda)$.
The simplicity assumption on $L$ implies $L=\uu(\lambda)$.
\end{proof}

By Theorem~\ref{prp:MS5.8} for $\T$,
we see that 
the category of locally finite left $\hy(\T)$-$T$-supermodules
is equivalent to 
the category of $\T$-supermodules.
Thus, we may regard $\uu(\lambda)$ as a $\T$-supermodule,
and hence we get the following map.
\[
\Lambda = \X(T) \longrightarrow \Simple_\Pi(\T);
\quad 
\lambda \longmapsto \uu(\lambda).
\]

\begin{prop} \label{prp:T-simple}
The above map is bijective.
Moreover, $\uu(\lambda)$ is of type $\mathtt{M}$
if and only if (1) $\delta_\lambda=0$ or 
(2) $d_\lambda$ is even and $\delta_\lambda\in(\Bbbk^\times)^2$.
\end{prop}

\begin{rem}
As a $\T$-supermodule, this $\uu(\lambda)$ is isomorphic to 
some copies of $\Bbbk^\lambda$ up to parity change,
where $\Bbbk^\lambda$ is the one-dimensional purely even 
left $\hy(T)$-supermodule of weight $\lambda$.
If $\h=\h_\eve$, then we have $\hy(\T)^\lambda=\Bbbk$,
and hence $\uu(\lambda)$ is just $\Bbbk^\lambda$ or $\Pi\, \Bbbk^\lambda$.
\qed
\end{rem}

By Proposition~\ref{prp:dim(u)},
we have the following.
\begin{prop}
If $\Bbbk$ is algebraically closed,
then 
the dimension of $\uu(\lambda)$
is given by $2^{\lfloor (d_\lambda+1)/2 \rfloor}$.
\end{prop}

\begin{rem} \label{rem:Serganova}
In \cite[\S9.2]{Ser11}, 
Serganova constructed
simple $\h$-supermodules $C_\lambda$ for each $\lambda\in \Lambda$,
when the base field is algebraically closed of characteristic zero.
Assume that our base field $\Bbbk$ is algebraically closed
(not necessarily $\charr(\Bbbk)=0$).
In the following,
we shall generalize her construction (in the supergroup level)
to our $\Bbbk$,
and show that our $\uu(\lambda)$ is isomorphic to her $C_\lambda$
up to parity change, when $\charr(\Bbbk)=0$.

Fix a maximal totally isotropic subspace 
$\h^\dagger_\odd$ of $(\h_\odd,b^\lambda)$.
We set
\[
\h^\dagger \,\,:=\,\, \h_\eve \oplus \h_\odd^\dagger.
\]
This forms a Lie super-subalgebra of $\h$.
It is obvious that the pair $(T,\h^\dagger)$ is a sub-pair of 
the pair $(T,\h)$,
and hence we get a closed super-subgroup
$\T^\dagger$ of $\T$.
In particular, $\hy(\T^\dagger)$
is isomorphic to $\hy(T) \otimes \wedge(\h_\odd^\dagger)$
as supercoalgebras.

We regard the one-dimensional purely even left $\hy(T)$-supermodule $\Bbbk^\lambda$ as
a left $\hy(\T^\dagger)$-supermodule
by letting $K.\Bbbk^\lambda=0$ for any $K\in\h_\odd^\dagger$.
Set
\begin{equation}
\coind_{\T^\dagger}^{\T}(\lambda)
\,\, := \,\,
\hy(\T) \otimes_{\hy(\T^\dagger)} \Bbbk^\lambda.
\end{equation}
By \eqref{eq:hy(T)-basis}, we see that
$\dim (\coind_{\T^\dagger}^\T(\lambda)) =
2^{\ell_\odd-\dim(\h_\odd^\dagger)}$.
Since the radical $\rad(b^\lambda)$ is contained in $\h_\odd^\dagger$,
one sees that $\coind_{\T^\dagger}^{\T}(\lambda)$ is a $\hy(\T)$-$T$-supermodule,
and hence it is a $\T$-supermodule.
Thus, $\coind_{\T^\dagger}^\T(\lambda)$ 
must contain $\uu(\lambda)$ or $\Pi\uu(\lambda)$, 
by Proposition~\ref{prp:T-simple}.
Then by \eqref{eq:dim(u)},
we conclude that
$\coind_{\T^\dagger}^\T(\lambda)$ is isomorphic to
$\uu(\lambda)$ up to parity change.
If $\charr(\Bbbk)=0$, then 
our $\coind_{\T^\dagger}^\T(\lambda)$ coincides with Serganova's $C_\lambda$
by Remark~\ref{rem:hy=U}.
\qed
\end{rem}

\subsection{Cohomology Groups and Induced Supermodules}
In this section,
we shall prepare some notations.
Let $\H$ be a closed super-subgroup of $\G$.
We consider the functor $(\G/\H)_{(\mathrm{n})} : R\mapsto \G(R)/\H(R)$
from the category of superalgebras
to the category of sets, 
which is called the {\it naive quotient} of $\G$ over $\H$.
Here, $\G(R)/\H(R)$ is the set of right cosets of $\H(R)$ in $\G(R)$,
as usual.
Then by Masuoka and Zubkov \cite[Theorem~0.1]{MasZub11},
the sheafification $\G/\H$ of the native quotient $(\G/\H)_{(\mathrm{n})}$
becomes a noetherian superscheme endowed with 
a morphism $\pi_{\G/\H}:\G\to\G/\H$ satisfying 
the conditions (Q1)--(Q5) described in Brundan's article \cite[\S2]{Bru06}.
See also Masuoka and Takahashi \cite[\S4.4]{MasTak18}.
As in \cite[\S2]{Bru06}
(see also \cite[Part~I, \S5.8]{Jan03}),
we let $\mathscr{L}=\mathscr{L}_{\G/\H}$ be 
a functor from ${}_{\H}\SMod$
to the category of quasi-coherent $\mathscr{O}_{\G/\H}\G$-supermodules
satisfying  
\[
\mathscr{L}(V)(\mathscr{U})
\,\,=\,\,
V \cotens_{\mathcal{O}(\H)}\mathcal{O}(\pi^{-1}_{\G/\H}(\mathscr{U}))
\]
for $V\in{}_{\H}\SMod$ and an open super-subfunctor $\mathscr{U}\subseteq \G/\H$
such that $\pi_{\G/\H}^{-1}(\mathscr{U})$ is affine.
This $\mathscr{L}(V)$ 
is the so-called {\it associated sheaf} to $V$ on $\G/\H$.

\medskip

Now, let us return to our situation.
Recall that, $G=\Gev$ and $B=\Bev$.
It is easy to see that
the quotient $\G/\B$ satisfies the condition (Q6) in \cite[\S2]{Bru06},
that is,
the even part $G/B=(\G/\B)_\ev$ of the quotient $\G/\B$ is projective.

Recall that, $\uu(\lambda)$ is a simple $\T$-supermodule
for $\lambda\in\Lambda$.
For simplicity,
we shall denote
the {\it cohomology group}
$H^n\big(\G/\B,\,\mathscr{L}(\res^{\T}_{\B}\uu(\lambda))\big)$
by $H^n(\lambda)$
for $\lambda\in\Lambda$ and $n\in\mathz_{\geq0}$.
Note that, $H^n(\lambda)$ does depend on the choice of 
the morphism $\gamma:\mathz\Delta\to \mathr$ defined in \S\ref{sec:tori,unip}.
Brundan \cite[Corollary~2.4]{Bru06} showed that
for each $n\geq0$,
we have an isomorphism
\[
H^n(\lambda) \,\,\cong\,\,
R^n\ind^{\G}_{\B}\big(\res^{\T}_{\B}(\uu(\lambda))\big)
\]
of $\G$-supermodules.
This is a super-analogue result of \cite[Part I, Proposition~5.12]{Jan03}.
In the following, we shall use this identification.
By construction, $\res^{\T}_{\B}(\uu(\lambda))$ is isomorphic to 
some (finite) copies of $\Bbbk^\lambda$
as $\B$-supermodules.
Thus, for $c\tens a \in \uu(\lambda)\tens \mathcal{O}(\G)$, we have
\[
c\tens a \in H^0(\lambda)
\iff
c\tens \lambda \tens a \,\,= \,\, c \tens \sum_a \overline{a_{(1)}}^{\B} \tens a_{(2)},
\]
where $\overline{a}^{\B}$ is the canonical image of $a\in \mathcal{O}(\G)$ in $\mathcal{O}(\B)$.

\subsection{Simple $\B$-Supermodules}
The inclusion $\U\subseteq \B$ makes $\mathcal{O}(\B)$ into 
a right $\mathcal{O}(\U)$-supercomodule.
We regard $\mathcal{O}(\T)\tens \mathcal{O}(\U)$ as a right $\mathcal{O}(\U)$-supercomodule
via $\id \tens \Deltaa_{\mathcal{O}(\U)}$.
Then one sees that the isomorphism 
$\mathcal{O}(\B)\cong \mathcal{O}(\T) \tens \mathcal{O}(\U)$ given in \eqref{eq:Bsplit}
is right $\mathcal{O}(\U)$-colinear.
By taking the functor $(-)\cotens_{\mathcal{O}(\U)}\Bbbk$ to both sides,
we get
\begin{equation} \label{eq:T=B^U}
\mathcal{O}(\B)^{\U} \,\,\cong\,\, \mathcal{O}(\T),
\end{equation}
where $\mathcal{O}(\B)^{\U}$ is the {\it $\U$-invariant super-subspace} of $\mathcal{O}(\B)$,
in other words, the $\mathcal{O}(\U)$-coinvariant super-subspace
$\mathcal{O}(\B)^{\co \mathcal{O}(\U)}$ of $\mathcal{O}(\B)$,
see Appendix~\ref{sec:supercomodules}.
It is easy to see that this isomorphism is left $\mathcal{O}(\B)$- right $\mathcal{O}(\T)$-colinear.

\medskip

In the following, for a left $\B$-supermodule $V$,
we consider $\res^{\B}_{\U}(V)$ and $\ind^{\T}_{\B}(V)$ via
the quotient maps $\mathcal{O}(\B)\twoheadrightarrow \mathcal{O}(\T)$ and
$\mathcal{O}(\B)\twoheadrightarrow \mathcal{O}(\U)$ respectively.

\begin{lemm}\label{prp:M^U=M^T_B}
For a left $\B$-supermodule $V$,
$\res^{\B}_{\U}(V)^{\U}$ is isomorphic to $\ind_{\B}^{\T}(V)$
as left $\T$-supermodules.
\end{lemm}

\begin{proof}
It is easy to see that 
the canonical isomorphism $M\cotens_{\mathcal{O}(\B)}\mathcal{O}(\B)\cong V$ is 
left $\mathcal{O}(\U)$-colinear.
Then by taking the functor $(-)\cotens_{\mathcal{O}(\U)}\Bbbk$ to both sides,
we get 
\[
\big(V\cotens_{\mathcal{O}(\B)}\mathcal{O}(\B)\big) \cotens_{\mathcal{O}(\U)}\Bbbk 
\,\, \cong \,\,
\res^{\B}_{\U}(V)^{\U}.
\]
The left hand side is isomorphic to $V\cotens_{\mathcal{O}(\B)} (\mathcal{O}(\B)^{\U})$.
Thus, by combining this with \eqref{eq:T=B^U}, we are done.
\end{proof}

For a left $\T$-supermodule $N$,
we consider $\res^{\T}_{\B}(N)$
via the Hopf superalgebra splitting $\spl_{\B}:\mathcal{O}(\T)\hookrightarrow \mathcal{O}(\B)$ 
given in \eqref{eq:section}.
Since the composition the splitting $\spl_{\B}$ with 
the quotient $\mathcal{O}(\B)\twoheadrightarrow \mathcal{O}(\T)$ is identity,
$\res^{\B}_{\T}(\res^{\T}_{\B}(N))$ is just the original $N$.

\begin{prop}\label{prp:res(u)-simple}
For $\lambda\in \Lambda$,
the left $\B$-supermodule $\res^{\T}_{\B} (\uu(\lambda))$ is simple.
Moreover, this gives a one-to-one correspondence between
$\Lambda$ and $\Simple_\Pi(\B)$.
\end{prop}

\begin{proof}
It is enough to show that for any simple $\B$-supermodule $L$,
there exists $\lambda\in\Lambda$ such that 
$L=\res^{\T}_{\B}(\uu(\lambda))$ or $\Pi \res^{\T}_{\B}(\uu(\lambda))$.
Since $L\not=0$, we see that $L^\U\not=0$,
by Proposition~\ref{prp:unip} and Lemma~\ref{prp:simple_lemma2}.
Then by Lemma~\ref{prp:M^U=M^T_B},
there exists $\lambda\in\Lambda$ such that 
$\ind^{\T}_{\B}(L)$ contains $\uu(\lambda)$ or $\Pi \uu(\lambda)$.
For simplicity, we shall concentrate on the case when 
$\uu(\lambda) \subseteq \ind^{\T}_{\B}(L)$.
Then by Frobenius reciprocity~\eqref{eq:Frob}, we get
\[
{}_{\B}\SMod\big(\res^{\T}_{\B}\big(\uu(\lambda)\big),\,L\big)
\,\,\cong\,\,
{}_{\T}\SMod\big(\uu(\lambda),\,\ind^{\T}_{\B}(L)\big)
\,\,\not=0.
\]
Hence, we get a $\B$-supermodule surjection 
$\res^{\T}_{\B}(\uu(\lambda))\twoheadrightarrow L$.
By applying the functor $\res^{\B}_{\T}(-)$ to both sides,
we have $\uu(\lambda) \twoheadrightarrow \res^{\B}_{\T}(L)$.
Since this is indeed a bijection, we are done.
\end{proof}

\subsection{Simple $\G$-Supermodules}
\label{sec:irrG}

Set
\[
\Lambda^\dominant \,\,:=\,\,\{\lambda\in \Lambda \mid H^0(\lambda) \not=0 \},
\qquad
L(\lambda)\,\,:=\,\,\soc_{\G}(H^0(\lambda)).
\]
Note that, for $\lambda\in\Lambda^\dominant$, 
the left $\G$-supermodule $L(\lambda)$ is non zero by Lemma~\ref{prp:simple_lemma}.

Let $N$ be a left $\T$-supermodule.
By taking the functor $\res^{\T}_{\B}(N)\cotens_{\mathcal{O}(\B)}(-)$ to 
the superalgebra inclusion $\mathcal{O}(\G) \hookrightarrow \mathcal{O}(\B)\tens \mathcal{O}(\B^+)$
given in \eqref{eq:inj},
we get 
\[
\res^{\T}_{\B}(N)\cotens_{\mathcal{O}(\B)}\mathcal{O}(\G) \hookrightarrow 
\res^{\T}_{\B}(N)\cotens_{\mathcal{O}(\B)}(\mathcal{O}(\B)\tens\mathcal{O}(\B^+)).
\]
Since the right hand side is equal to $N\tens \mathcal{O}(\B^+)$,
we have $\ind^{\G}_{\B}(\res^{\T}_{\B}(N)) \hookrightarrow N\tens \mathcal{O}(\B^+)$.
It is easy to see that the image of the above map lies in $N\cotens_{\mathcal{O}(\T)}\mathcal{O}(\B^+)$.
Thus, we get an inclusion
\begin{equation}\label{eq:indres->ind}
\res^{\G}_{\B^+}(\ind^{\G}_{\B}(\res^{\T}_{\B}(N)))
\hookrightarrow
\ind^{\B^+}_{\T}(N)
\end{equation}
of right $\mathcal{O}(\B^+)$-supercomodules, or equivalently left $\B^+$-supermodules.

\medskip

For simplicity,
we write $\res^{\G}_{\B^+}(H^0(\lambda))$ as $H^0(\lambda)$ for $\lambda\in\Lambda$.
Then by \eqref{eq:indres->ind}, 
$H^0(\lambda)$ can be regarded as a $\B^+$-super-submodule of 
$\ind^{\B^+}_{\T}(\uu(\lambda))$.

\begin{lemm}\label{prp:socH^0}
For $\lambda\in \Lambda^\dominant$,
we have $\soc_{\B^+}(H^0(\lambda))=\res^{\T}_{\B^+}(\uu(\lambda))$.
\end{lemm}

\begin{proof}
We see that $\soc_{\B^+}(H^0(\lambda))$ is contained in 
$\soc_{\B^+}\big(\ind^{\B^+}_{\T}(\uu(\lambda))\big)$.
To conclude the proof,
it is enough to show that
\begin{equation} \label{eq:soc=res}
\soc_{\B^+}\big(\ind^{\B^+}_{\T}(\uu(\lambda))\big) \,\,=\,\, 
\res^{\T}_{\B^+}(\uu(\lambda)).
\end{equation}
By Proposition~\ref{prp:res(u)-simple},
any simple $\B^+$-super-submodule of $\ind^{\B^+}_{\T}(\uu(\lambda))$ 
is ether (i) $\res^{\T}_{\B^+}(\uu(\mu))$ or (ii) $\Pi \res^{\T}_{\B^+}(\uu(\mu))$
for some $\mu\in\Lambda$.
First, we consider the case (i).
In this case, we have
\[
0\not=\,\,
{}_{\B^+}\SMod\big(\res^{\T}_{\B^+}\big(\uu(\mu)\big),\,
\ind^{\B^+}_{\T}(\uu(\lambda))\big)
\,\,\cong\,\,
{}_{\T}\SMod\big(\uu(\mu),\,\uu(\lambda)\big),
\]
by Frobenius reciprocity \eqref{eq:Frob} and 
$\ind^{\T}_{\B^+}(\ind^{\B^+}_{\T}(\uu(\lambda)))=\uu(\lambda)$.
Thus, we conclude that $\lambda=\mu$,
and hence the equation \eqref{eq:soc=res} holds.
Next, we consider the case (ii).
Similarly, by Frobenius reciprocity, we have 
$\Pi\uu(\mu)\cong \uu(\lambda)$.
Thus, we conclude that $\mu=\lambda$ and $\uu(\lambda)$ is of type $\QQ$,
and hence the equation \eqref{eq:soc=res} also holds.
\end{proof}

\begin{prop}\label{prp:L(lambda)}
For $\lambda\in \Lambda^\dominant$,
the left $\G$-supermodule $L(\lambda)$ is a unique simple super-submodule of $H^0(\lambda)$.
\end{prop}

\begin{proof}
Suppose that $L,L'$ are two simple $\G$-super-submodules of $H^0(\lambda)$.
Then by Lemma~\ref{prp:socH^0}, 
$\soc_{\B^+}(\res^{\G}_{\B^+}(L))$ and 
$\soc_{\B^+}(\res^{\G}_{\B^+}(L'))$ should coincide with $\res^{\T}_{\B^+}(\uu(\lambda))$.
Thus, $\uu(\lambda)$ is included in $L\cap L'$, and hence $L=L'$.
\end{proof}

By Proposition~\ref{prp:L(lambda)},
we get the following map.
\[
\Lambda^\dominant \longrightarrow \Simple_\Pi(\G);
\qquad
\lambda \longmapsto L(\lambda).
\]
The next is our main theorem,
which is a generalization of Serganova's result \cite[Theorem~9.9]{Ser11}.

\begin{thm}\label{prp:LMQ}
The map above is bijective.
Moreover, if (1) $\delta_\lambda=0$ or 
(2) $d_\lambda$ is even and $\delta_\lambda\in(\Bbbk^\times)^2$,
then $L(\lambda)$ is of type $\MM$.
Otherwise, $L(\lambda)$ is of type $\QQ$.
\end{thm}

\begin{proof}
Let $L$ be a simple $\G$-supermodule.
We see that $\soc_{\B}(\res^{\G}_{\B}(L^*))\not=0$
by Lemma~\ref{prp:simple_lemma}.
Then 
$\res^{\G}_{\B}(L^*)$ includes 
$\res^{\T}_{\B}(\uu(\mu))$ or $\Pi\res^{\T}_{\B}(\uu(\mu))$
for some $\mu\in\Lambda$.
Taking the linear dual $(-)^*$,
we get a morphism from
$\res^{\G}_{\B}(L)$ to $\res^{\T}_{\B}(\uu(\mu))^*$ or $\Pi\res^{\T}_{\B}(\uu(\mu))^*$.
Here, we used $L=L^{**}$.
Since $\uu(\mu)^*$ is simple,
there exists $\lambda\in\Lambda$ such that $\uu(\mu)^*=\uu(\lambda)$ 
by Proposition~\ref{prp:res(u)-simple}.
Thus, we get a non trivial morphism from 
$\res^{\G}_{\B}(L)$ to (i) $\res^{\T}_{\B}(\uu(\lambda))$ or 
(ii) $\Pi\res^{\T}_{\B}(\uu(\lambda))$.

First, we consider the case (i).
By Frobenius reciprocity \eqref{eq:Frob}, we get
\[
0\not=\,\,
{}_{\B}\SMod\big(\res^{\G}_{\B}(L),\,\res^{\T}_{\B}(\uu(\lambda))\big)
\,\,\cong\,\,
{}_{\T}\SMod\big(L,\,H^0(\lambda)\big).
\]
Thus, we conclude that $L\hookrightarrow H^0(\lambda)$, 
and hence $L=L(\lambda)$.
Next, we consider the case (ii).
In this case, a similar argument ensures that 
$L\hookrightarrow \Pi H^0(\lambda)$, and hence $L=\Pi L(\lambda)$.

Since $\Pi L(\lambda) = \soc_G(\Pi H^0(\lambda))$,
we conclude that $L(\lambda)$ is of type $\MM$ (resp. $\QQ$)
if and only if $\uu(\lambda)$ is of type $\MM$ (resp. $\QQ$).
Thus, the last statement directly follows from Proposition~\ref{prp:T-simple}.
\end{proof}

If $\h_\odd=0$,
then $b^\lambda=0$ by definition. 
Hence, in this case, $L(\lambda)$ is always of type $\MM$,
that is,
$L(\lambda)$ is not isomorphic to $\Pi L(\lambda)$
as $\G$-supermodules.

\begin{rem}
For the queer supergroup $\G=\Q(n)$,
the result above is a part of 
Brundan and Kleshchev's result \cite[Theorem~6.11]{BruKle03}
(where the base field $\Bbbk$ is assumed to be algebraically closed).
Indeed, one easily sees that 
their $h_{p'}(\lambda)$ coincides with our $d_\lambda$,
where
$p:=\charr(\Bbbk)$ and
$h_{p'}(\lambda):=\#\{i\in\{1,\dots,n\} \mid p\nmid d_i\}$
for $\lambda = d_1\lambda_1 + \cdots + d_n\lambda_n 
\in \Lambda \cong \bigoplus_{i=1}^n \mathz \lambda_i$.
\qed
\end{rem}

\subsection{Some Properties of Induced Supermodules}
\label{sec:someH^0}

In this section, we study some properties of 
$H^0(\lambda)$ for $\lambda\in\Lambda^\dominant$.

\begin{lemm}\label{prp:H^0^U}
For $\lambda\in \Lambda^\dominant$,
we have $H^0(\lambda)^{\U^+} \cong \uu(\lambda)$ as left $\T$-supermodules.
\end{lemm}

\begin{proof}
By taking the functor 
$\ind^{\T}_{\B^+}(-)$ to the inclusion 
$H^0(\lambda) \hookrightarrow \ind^{\B^+}_{\T}(\uu(\lambda))$,
we get an inclusion
$\ind^{\T}_{\B^+}(H^0(\lambda)) \hookrightarrow \uu(\lambda)$
of $\T$-supermodules,
and hence this is bijective.
Then by a $\B^+$-version of Lemma~\ref{prp:M^U=M^T_B}, we are done.
\end{proof}

For $\lambda\in\Lambda^\dominant$,
we regard $H^0(\lambda)$ as a left $T$-supermodule
via the left adjoint action, as usual.
Then for $\mu\in\Lambda$ ($\subseteq \mathcal{O}(T)$),
the $\mu$-weight superspace of $H^0(\lambda)$ is described as follows.
\[
H^0(\lambda)^\mu \,\,=\,\,
\{ \xi \in H^0(\lambda) \mid u\hits \xi = \langle u, \mu\rangle \xi 
\,\,\text{ for } u\in\hy(T) \}.
\]
We define a partial order $\leq$ on $\Lambda$
as follows.
\begin{equation} \label{eq:order}
\mu \leq \lambda \quad :\iff\quad
\lambda-\mu =
\sum_{\alpha\in\Delta^+} n_\alpha \alpha 
\quad \text{for some } n_\alpha\in\mathz_{\geq0}.
\end{equation}
The following proof is based on the proof of
\cite[Part~II, Proposition~2.2(a)]{Jan03}.

\begin{prop}\label{prp:maxwt}
For $\lambda\in\Lambda^\dominant$,
$\lambda$ is a maximal $T$-weight of $H^0(\lambda)$ with respect to $\leq$
and $H^0(\lambda)^\lambda \cong \uu(\lambda)$ as left $\T$-supermodules.
\end{prop}

\begin{proof}
Suppose that $\mu\in\Lambda$ is a maximal $T$-weight of $H^0(\lambda)$.
Then by definition, the $\mu$ weight superspace $H^0(\lambda)^\mu$
is included in the $\U^+$-invariant super-subspace 
$H^0(\lambda)^{\U^+}$ of $H^0(\lambda)$.

On the other hand, 
since $\uu(\lambda)$ is isomorphic to some copies of $\Bbbk^\lambda$ as $T$-modules,
we see that $H^0(\lambda)^\lambda\not=0$ by Lemma~\ref{prp:H^0^U}.
We fix a non-zero element $x\in H^0(\lambda)^\lambda$ and 
consider the following map.
\[
H^0(\lambda) \longrightarrow H^0(\lambda)^\lambda;
\quad c\tens a \longmapsto c\tens \varepsilon_\G(a) x,
\]
where $c\tens a \in\uu(\lambda)\tens \mathcal{O}(\G)$ and 
$\varepsilon_\G : \mathcal{O}(\G)\to \Bbbk$ is the counit of $\mathcal{O}(\G)$.
We will show that this is injective.
For simplicity, we assume that $\uu(\lambda)=\Bbbk^\lambda$.
For $c\tens a\in H^{0}(\lambda)^{\U^+}$ with $c\not=0$ and $\varepsilon_\G(a)=0$,
it is easy to see that the canonical images of $a\in\mathcal{O}(\G)$ 
in $\mathcal{O}(\U^+)$ and $\mathcal{O}(\B)$ are both zero.
As in the proof of Corollary~\ref{prp:inj},
we can also prove that 
the composition $\mathcal{O}(\G)\to \mathcal{O}(\G)\tens \mathcal{O}(\G) \to \mathcal{O}(\U^+)\tens \mathcal{O}(\B)$
is injective.
Thus, we can conclude that $a=0$,
and hence we may regard $H^0(\lambda)^{\U^+}$ as 
a $\T$-super-submodule of $H^0(\lambda)^\lambda$.
Combined with the above, 
we get $\mu=\lambda$ and $H^0(\lambda)^{\U^+}=H^0(\lambda)^\lambda$.
\end{proof}

Set $A:=\mathcal{O}(\G)$.
Recall that, 
$I_A$ is the super-ideal of $A$ generated by the odd part $A_\odd$, 
see \S\ref{sec:alg}.
Then the (finite) descending chain 
$I^0_A := A \supseteq I_A \supseteq I^2_A \supseteq \cdots$ defines
the graded (ordinary) algebra 
\[
\gr(A) \,\, := \,\, \bigoplus_{n\geq 0} I^n_A/I^{n+1}_A
\]
of $A$.
On the other hand,
we regard $A$ as a right $\overline{A}$-comodule via
the right coadjoint action \eqref{eq:coad}.
This coaction makes $\wedge(W^{A})$ into a right $\overline{A}$-comodule Hopf algebra,
and hence we can construct the cosmash product 
$\overline{A} \cmdblackltimes \wedge(W^{A})$.
Then we have an isomorphism 
$\gr(A) \cong \overline{A} \cmdblackltimes \wedge(W^{A})$
of graded (ordinary) Hopf algebras,
see \cite[Proposition~4.9(2)]{Mas05}.
Recall that, $T$ is a split maximal torus of $G$.
Since $\mathcal{O}(T)$ is a cosemisimple Hopf algebra, 
we see that the graded algebra $\gr(A)$ is isomorphic to
$A$ as right $\mathcal{O}(T)$-comodules via the adjoint action of $T$.
Hence, we get an isomorphism
\begin{equation} \label{eq:A=A0(tens)W}
A \overset{\cong}{\longrightarrow} 
\overline{A} \cmdblackltimes \wedge (W^{A})
\end{equation}
of left $\overline{A}$- right $\mathcal{O}(T)$-comodules
(or equivalently, right $G$- left $T$-modules).
Here, the left $\overline{A}$-comodule structure of right hand side is 
$\Deltaa_{\overline{A}} \tens \id$.

Recall that $G=\Gev$ and $B=\Bev$.
Then we have the following lemma.

\begin{lemm}\label{prp:ind->ind(tens)wedge}
For a left $\B$-supermodule $V$, there is an inclusion
\[
\ind^{\G}_{\B}(V) \hookrightarrow 
\ind^{G}_{B}(\res^{\B}_B(V)) \tens \wedge(W^{\mathcal{O}(\G)})
\]
of left $T$-modules.
\end{lemm}

\begin{proof}
Taking the functor $V\cotens_{\mathcal{O}(B)}\res^{G}_{B}(-)$ 
to both sides in \eqref{eq:A=A0(tens)W},
we get
$V\cotens_{\mathcal{O}(B)} A \overset{\cong}{\to} 
(V\cotens_{\mathcal{O}(B)} \overline{A}) \tens \wedge(W^A)$.
By the definition of the cotensor,
we see that $V\cotens_{\mathcal{O}(\B)} A$ 
is contained in $V\cotens_{\mathcal{O}(B)} A$.
This completes the proof.
\end{proof}

\begin{prop}\label{prp:dim(ind)<fin}
For a finite-dimensional left $\B$-supermodule $V$,
the induced supermodule $\ind^{\G}_{\B}(V)$ is finite-dimensional.
In particular, $H^0(\lambda)$ is finite for any $\lambda\in\Lambda$.
\end{prop}

\begin{proof}
Since $\mathcal{O}(\G)$ is a finitely generated superalgebra,
$W^{\mathcal{O}(\G)}$ is finite-dimensional
by \cite[Proposition~4.4]{Mas05}.
On the other hand, it is known that $\ind^G_B(V)$ is finite dimensional,
see \cite[Part~I, 5.12(c)]{Jan03} for example.
This claim follows immediately from Lemma~\ref{prp:ind->ind(tens)wedge}.
\end{proof}

Recall that, $G=\Gev$ and $B=\Bev$.
As a non-super version of $H^0(\lambda)$, we set 
\[
H^0_\ev(\lambda):=\ind^{G}_{B}(\Bbbk^\lambda)
\quad \text{and} \quad
\Lambda^+:=\{\lambda\in\Lambda \mid H^0_\ev(\lambda)\not=0\}.
\]
for $\lambda\in\Lambda$.
It is known that $\Lambda^+$ consists of all {\it dominant weights}
of $T$ with respect to $\Delta_\eve^+$,
that is,
$\Lambda^+ = \{
\lambda\in\Lambda\mid \langle\lambda,\alpha^\vee \rangle \geq 0
\text{ for all } \alpha\in\Delta^+_\eve\}$,
see \cite[Part~II, 2.6]{Jan03} for example.
Then by Lemma~\ref{prp:ind->ind(tens)wedge},
we get the following Proposition.

\begin{prop}\label{prp:H^0<H^0_ev(tens)wedge}
For $\lambda\in\Lambda$, there is an inclusion
\[
H^0(\lambda) \hookrightarrow 
H^0_\ev(\lambda)^{\oplus n_\lambda} \tens \wedge (W^{\mathcal{O}(\G)})
\]
of left $T$-modules,
where $n_\lambda:=\dim(\uu(\lambda))$.
\end{prop}

In particular, we see that the parameter set $\Lambda^\dominant$ 
is included in $\Lambda^+$.

\begin{rem}
Assume that $\Bbbk$ is algebraically closed.
Brundan and Kleshchev \cite[Theorem~6.11]{BruKle03}
determined $\Lambda^\dominant$ explicitly
for queer supergroups $\Q(n)$; $\Lambda^\dominant=X_p^+(n)$ in their notation.
Shu and Wang \cite[Theorem~5.3]{ShuWan08}
describe $\Lambda^\dominant$ combinatorially
for ortho-symplectic supergroups $\SpO(m|n)$;
$\Lambda^\dominant=\mathscr{X}^\dagger(T)$ in their notation.
Recently, Cheng, Shu and Wang \cite{CheShuWan18}
explicitly determined $\Lambda^\dominant$
for simply connected Chevalley groups of type
$D(2|1;\zeta),G(3)$ and $F(3|1)$.
\qed
\end{rem}

A {\it unified} description of the parameter set $\Lambda^\dominant$ 
for all quasireductive supergroups,
as in the case of (non-super) split reductive groups,
is not known.
This is an {\it open problem}.

\section{Quasireductive Supergroups admitting Distinguished Parabolic Super-subgroups}
\label{sec:parab}

Let $\G$ be a quasireductive supergroup over a field $\Bbbk$, as before.
Recall that, $G$ is the even part $\Gev$ of $\G$ and 
$\B$ is the Borel super-subgroup of $\G$
satisfying $\bb^- = \Lie(\B)$.
In this section,
we assume that $\G$ contains a closed super-subgroup $\P$ 
such that 
\begin{equation} \label{eq:parab}
\Pev\,\,=\,\,G
\quad \text{and}
\quad
\Lie(\P)_\odd \,\,=\,\, \bb^-_\odd,
\end{equation}
which we shall call a {\it maximal parabolic super-subgroup} of $\G$.
Note that, the assumption does depend on 
the choice of the homomorphism $\gamma:\mathz\Delta\to \mathr$,
defined in \S\ref{sec:tori,unip}.

\begin{exs}
The following, with a suitable choice of $\gamma$, satisfy the assumptions above.
\begin{enumerate}
\item
General linear supergroups $\GL(m|n)$,
see Example~\ref{ex:GL(m|n)}.
\item
Chevalley supergroups $\G$ of classical type
such that the Lie superalgebra $\g=\Lie(\G)$ of $\G$ is 
a finite-dimensional simple Lie superalgebra of type I.
\item
Periplectic supergroups $\mathbf{P}(n)$, see Example~\ref{ex:P(n)}.
\end{enumerate}
However, queer supergroups $\Q(n)$ do not satisfy the assumption.
\qed
\end{exs}

For simplicity, we also assume that $\h_\odd=0$,
or equivalently, $\T=T$.

\subsection{A $\tens$-Splitting Property}
\label{sec:tens-split}

First, we fix a totally ordered $\Bbbk$-basis $\mathscr{G}=((X_i)_i,\leq)$ of $\g_\odd$
and define the following unit-preserving supercoalgebra map
\[
\iota_{\mathscr{G}} : \wedge(\g_\odd) \longrightarrow \hy(\G);
\quad 
X_{i_1} \wedge X_{i_2} \wedge \cdots \wedge X_{i_r} 
\longmapsto X_{i_1} X_{i_2} \cdots X_{i_r},
\]
where $X_{i_1}<X_{i_2}<\cdots <X_{i_r}$ in $\mathscr{G}$.
The left $T$-supermodule structure on $\g_\odd$
naturally induces a left $T$-supermodule structure on 
the exterior superalgebra $\wedge(\g_\odd)$.
It is easy to see that $\iota_{\mathscr{G}}$ is left $T$-linear.
Then one can retake the isomorphism \eqref{eq:phi} so that
\[
\phi_{\G,\mathscr{G}} : 
\hy(G) \tens \wedge(\g_\odd) \longrightarrow \hy(\G);
\quad u\tens \xi \longmapsto u \cdot \iota_{\mathscr{G}}(\xi).
\]
Moreover, by dualizing $\iota_{\mathscr{G}}$,
we get a counit-preserving left $T$-module superalgebra map
$\pi_{\mathscr{G}} : \mathcal{O}(\G) \to \wedge (W^{\mathcal{O}(\G)})$ satisfying
\[
\langle \iota_{\mathscr{G}}(\xi),\, \eta \rangle
\,\,=\,\,
\langle \xi,\, \pi_{\mathscr{G}}(\eta) \rangle,
\qquad \xi\in \wedge(\g_\odd),\,\, \eta\in\wedge(W^{\mathcal{O}(\G)})
\]
and the isomorphism \eqref{eq:psi} can be chosen as follows.
\[
\psi_{\G,\mathscr{G}} : 
\mathcal{O}(\G) \longrightarrow \mathcal{O}(G) \tens \wedge(W^{\mathcal{O}(\G)});
\quad a \longmapsto \sum_a \overline{a_{(1)}} \tens \pi_{\mathscr{G}}(a_{(2)}),
\]
see \cite[Proposition~22]{Mas12}.
It is easy to see that this $\psi_{\G,\mathscr{G}}$ is also left $T$-linear.

Take $\mathscr{B},\mathscr{U}^+ \subseteq \mathscr{G}$ so that
$\mathscr{G}=\mathscr{B}\sqcup\mathscr{U}^+$ (disjoint union) and
$\mathscr{B}$ (resp. $\mathscr{U}^+$) 
forms a $\Bbbk$-basis of $\bb^-_\odd$ (resp. $\n^+_\odd$).
Then we also get $\pi_{\mathscr{B}} : \mathcal{O}(\B) \to \wedge(W^{\mathcal{O}(\B)})$
and $\pi_{\mathscr{U}^+} : \mathcal{O}(\U^+) \to \wedge(W^{\mathcal{O}(\U^+)})$.
The canonical identification $\g = \bb^- \oplus \uu^+$ induces an isomorphism 
$\wedge(W^{\mathcal{O}(\G)}) \overset{\cong}{\to} 
\wedge(W^{\mathcal{O}(\B)})\tens \wedge(W^{\mathcal{O}(\U^+)})$.
One sees that the following diagram commutes.
\begin{equation} \label{eq:pi_comm}
\begin{xy}
(0,7)   *++{\mathcal{O}(\G)}  ="1",
(60,7)  *++{\mathcal{O}(\B)\tens \mathcal{O}(\U^+)}    ="2",
(0,-7) *++{\wedge(W^{\mathcal{O}(\G)})}          ="3",
(60,-7)*++{\wedge(W^{\mathcal{O}(\B)})\tens \wedge(W^{\mathcal{O}(\U^+)}),}    ="4",
(30,-1)*{\circlearrowleft},
{"1" \SelectTips{cm}{} \ar @{^(->}^{} "2"},
{"1" \SelectTips{cm}{} \ar @{->>}_{\pi_{\mathscr{G}}\,} "3"},
{"2" \SelectTips{cm}{} \ar @{->>}^{\pi_{\mathscr{B}} \tens \pi_{\mathscr{U}^+}} "4"},
{"3" \SelectTips{cm}{} \ar @{->}_{\cong} "4"}
\end{xy}
\end{equation}
where the upper arrow is given by \eqref{eq:inj}.

In the following, 
we regard $\mathcal{O}(\G)$ as a right $\P$-supermodule
via the canonical quotient map
$\varpi_\P:\mathcal{O}(\G) \twoheadrightarrow \mathcal{O}(\P)$,
and regard $\mathcal{O}(\P)$ and $\wedge(W^{\mathcal{O}(\U^+)})$ as left $T$-supermodules, as before.
Then we have the following.

\begin{prop} \label{prp:P-split}
$\mathcal{O}(\G)$ is isomorphic to $\mathcal{O}(\P)\tens \wedge(W^{\mathcal{O}(\U^+)})$
as right $\P$- left $T$-supermodules.
\end{prop}

\begin{proof}
By the definition of $\P$, 
we also have a counit-preserving
left $\mathcal{O}(G)$- right $\mathcal{O}(T)$-comodule superalgebra isomorphism
$\psi_{\P,\mathscr{B}} : \mathcal{O}(\P) \to \mathcal{O}(G) \tens \wedge(W^{\mathcal{O}(\B)})$
which satisfies
\[
\psi_{\P,\mathscr{B}}(\varpi_\P(a)) \,\,= \,\, 
\sum_a \overline{a_{(1)}} \tens \pi_{\mathscr{B}}(\varpi_{\B}(a_{(2)})),
\qquad a\in \mathcal{O}(\G),
\]
where $\varpi_\B : \mathcal{O}(\G)\to \mathcal{O}(\B)$ is the canonical quotient map.
Then by \eqref{eq:pi_comm}, we get the following commutative diagram
in the category of left $T$-supermodules.
\[
\begin{xy}
(0,10)   *++{\mathcal{O}(\G)}  ="1",
(80,10)  *++{\mathcal{O}(G)\tens \wedge(W^{\mathcal{O}(\G)})}    ="2",
(0,-17) *++{\mathcal{O}(\P) \tens \wedge(W^{\mathcal{O}(\U^+)})}          ="3",
(0,-3) *++{\mathcal{O}(\G) \tens \wedge(W^{\mathcal{O}(\G)})}          ="5",
(80,-17)*++{\mathcal{O}(G)\tens \wedge(W^{\mathcal{O}(\B)}) \tens \wedge(W^{\mathcal{O}(\U^+)}).}    ="4",
(40,-3)*{\circlearrowleft},
{"1" \SelectTips{cm}{} \ar @{->}^{\cong}_{\psi_{\G,\mathscr{G}}} "2"},
{"1" \SelectTips{cm}{} \ar @{->}^{(\id\otimes \pi_{\mathscr{G}}) \circ \Deltaa_{\mathcal{O}(\G)}} "5"},
{"5" \SelectTips{cm}{} \ar @{->}^{\varpi_{\P} \otimes\wedge f} "3"},
{"2" \SelectTips{cm}{} \ar @{->}^{\cong} "4"},
{"3" \SelectTips{cm}{} \ar @{->}_{\psi_{\P,\mathscr{B}}\tens \id}^{\cong} "4"}
\end{xy}
\]
Here, $f:W^{\mathcal{O}(\G)}\to W^{\mathcal{O}(\U^+)}$ is the canonical projection
induced by the inclusion $\mathscr{U}^+\subset \mathscr{G}$.
Thus, we are done.
\end{proof}

Recall that, $G=\Gev$ and $B=\Bev$.
For a left $\B$-supermodule $V$,
it is easy to see that 
the map $\id\tens \varepsilon_{\P} : V\cotens_{\mathcal{O}(\B)} \mathcal{O}(\P) \to V$
is left $B$-linear,
where $\varepsilon_{\P}:\mathcal{O}(\P)\to \Bbbk$ is the counit of $\mathcal{O}(\P)$.
Then by Frobenius reciprocity \eqref{eq:Frob},
we have the following left $G$-module map
\[
\mathcal{N}_V : 
V\cotens_{\mathcal{O}(\B)} \mathcal{O}(\P) \longrightarrow V\cotens_{\mathcal{O}(B)}\mathcal{O}(G);
\quad v\tens p \longmapsto v\tens \overline{p},
\]
where $\overline{p}$ is the canonical image of $p\in \mathcal{O}(\P)$ in $\mathcal{O}(G)$.

\begin{lemm} \label{prp:Zub5.2}
The above 
$\mathcal{N}:\res^{\P}_{G}\ind^{\P}_{\B}(-) \to \ind^{G}_{B}\res^{\B}_{B}(-)$
is a natural isomorphism.
\end{lemm}

\begin{proof}
The proof is essentially the same as Zubkov's \cite[Proposition~5.2]{Zub06}.
For simplicity, we write $R:=\mathcal{O}(\B)$, $\overline{R}:=\mathcal{O}(B)$, $P:=\mathcal{O}(\P)$.
Note that, $\overline{A}=\mathcal{O}(G)=\mathcal{O}(\Pev)$.
First, we prove that $\mathcal{N}_{R}$ is an isomorphism.
By the right version of \eqref{eq:psi}, we see that
$R$ is isomorphic to $\wedge(W^{R})\tens \overline{R}$
as a right $\overline{R}$-supercomodule.
Thus, we have
\[
\mathcal{N}_{R} : 
R\cotens_{R} P \longrightarrow R\cotens_{\overline{R}} \overline{A}
\,\,\cong\,\, {}^{\co \overline{R}}R \tens \overline{A}.
\]
Here, ${}^{\co\overline{R}}R$ is the right $\overline{R}$-coinvariant subspace of $R$,
that is, ${}^{\co\overline{R}}R=\Bbbk\cotens_{\overline{R}}R$,
where $\Bbbk$ is regarded as the trivial one-dimensional right $\overline{R}$-comodule.
It is known that ${}^{\co \overline{R}}R$ is isomorphic to $\wedge(W^R)$,
see \cite[Theorem~4.5]{Mas05}.
Then one sees that this map coincides with 
the tensor decomposition $P\cong \wedge(W^{R})\tens \overline{A}$ of $P$,
which is the right version of $\psi_{\P,\mathscr{B}}$,
and hence $\mathcal{N}_{B}$ is an isomorphism.
We see that $\mathcal{N}_{\Pi B}$ is also an isomorphism.

Next, we prove that $\mathcal{N}_V$ is an isomorphism 
for all injective object $V$ in $\SMod^B$.
By \cite[Proposition~3.1]{Zub06}, it is shown that 
$V$ is a direct summand of a direct sum of some copies of $B$ and $\Pi B$.
Thus, $\mathcal{N}_V$ is an isomorphism.
Then by \cite[Lemma~8.4.5]{ParWan91},
we are done.
\end{proof}

The following is a generalization of 
Zubkov's result \cite[Proposition~5.2]{Zub06}.

\begin{thm} \label{prp:R-decomp}
For a left $\B$-supermodule $V$,
there is an isomorphism
\[
R^n \ind^{\G}_{\B}(V) \,\, \cong\,\,
R^n\ind^{G}_{B}(\res^{\B}_{B}(V)) \tens \wedge (W^{\mathcal{O}(\U^+)})
\]
of left $T$-supermodules.
\end{thm}

\begin{proof}
Since $\res^{\P}_{G}(-)$ is exact,
we get $\res^{\P}_{G} \circ R^n \ind^{\P}_{\B}(-) 
\cong R^n(\res^{\P}_{G}\ind^{\P}_{\B}(-))$,
see \cite[Part~I, 4.1(2)]{Jan03}.
Then by Lemma~\ref{prp:Zub5.2}, we see that 
$\res^{\P}_{G}\circ R^n \ind^{\P}_{\B}(-) \cong R^n(\ind^G_B\res^{\B}_B(-))$.
By Proposition~\ref{prp:P-split},
we see that $\mathcal{O}(\G)$ is an injective object
in the category of left $\mathcal{O}(\P)$-supercomodules,
and hence the functor $\ind^\G_\P(-)$ is exact.
Thus, again by \cite[Part~I, 4.1(2)]{Jan03}, we get
$\ind^\G_\P\circ R^n(\ind^\P_\B(-)) \cong R^n\ind^\G_\B(-)$,
since $\ind^\G_\B(-)=\ind^\G_\P\ind^\P_\B(-)$.
On the other hand,
for any right $\mathcal{O}(\P)$-supercomodule $Y$, we get
an isomorphism
\[
\ind^\G_\P (Y)\,\,=\,\, Y\cotens_{\mathcal{O}(\P)}\mathcal{O}(\G)
\,\,\cong\,\, Y \tens \wedge(W^{\mathcal{O}(\U^+)})
\]
of left $\mathcal{O}(T)$-supercomodules
by Proposition~\ref{prp:P-split}.
Thus, we are done.
\end{proof}

\subsection{Some Applications}
\label{sec:final-appl}

As before, we set 
$H^n_\ev(\lambda) := R^n\ind^G_B(\res^T_B (\Bbbk^\lambda))$
for $\lambda\in\Lambda$ and $n\in\mathz_{\geq0}$.
By Theorem~\ref{prp:R-decomp},
we get the following theorem.

\begin{thm} \label{prp:main}
For $\lambda\in\Lambda$ and $n\geq 0$,
there is an isomorphism 
\[
H^n(\lambda) \,\,\cong \,\,
H^n_\ev(\lambda) \tens \wedge(W^{\mathcal{O}(\U^+)})
\]
of left $T$-supermodules.
In particular, we have $\Lambda^\dominant = \Lambda^+$.
\end{thm}

This seems to be a certain dual version of the notion of {\it Kac modules}
(cf. \cite[\S2.2~(b)]{Kac77-3}).
As a corollary, we get 
a super-analogue of the Kempf vanishing theorem
\cite[Part~II, Proposition~4.5]{Jan03}.

\begin{cor} \label{prp:cor1}
For $\lambda\in\Lambda^+$ and $n\geq 1$, we have $H^n(\lambda)=0$.
\end{cor}

The result above is first proved by 
Zubkov \cite[Theorem~5.1]{Zub06} for $\G=\GL(m|n)$
with the {\it standard} Borel super-subgroup.
Later he also gave
a partial generalization of the Kempf vanishing theorem for 
$\G=\GL(m|n)$
with all Borel super-subgroups \cite[Proposition~13.3]{Zub16}
(i.e., for all homomorphism $\gamma:\mathz\Delta\to \mathr$).

\medskip

Set $\rho_\eve:= \frac12 \sum_{\alpha\in\Delta_\eve^+}\alpha$
in $\Lambda\tens_\mathz\mathq$,
and 
\[
\Lambda^+_p
:=
\begin{cases}
\Lambda^+
& \text{if } p=0,\\
\{\lambda\in\Lambda^+ \mid 0\leq \langle \lambda+\rho_\eve,\, \beta^\vee \rangle\leq p,
\,\,\forall \beta\in\Delta_\eve^+ \}
& \text{if } p>2.
\end{cases}
\]
Here, we put $p:=\charr(\Bbbk)$.
Then for $\lambda\in\Lambda^+_p$,
it is known that the induced $G$-module $H^0_\ev(\lambda)$ is simple,
that is, $H^0_\ev(\lambda)=L_\ev(\lambda):=\soc_G(H^0_\ev(\lambda))$.
See \cite[Part II, 5.5]{Jan03}.
Thus, in this case, we see that
$H^0(\lambda)$ is isomorphic to the direct sum of $2^{\dim\g_\odd}$-copies of 
$L_\ev(\lambda)$ as $T$-modules.

\begin{rem}
For the case of $\G=\GL(m|n)$,
Marko \cite{Mar18} gave a necessary and sufficient condition
for the induced $\G$-supermodule $H^0(\lambda)$ to be simple,
see also Zubkov \cite[Proposition~12.10]{Zub16}.
\qed
\end{rem}

Let $\mathz\Lambda$ be the group algebra of $\Lambda$ over $\mathz$,
and let $\{\e^\lambda \mid \lambda\in \Lambda\}$ be the 
standard basis of $\mathz\Lambda$.
For a finite-dimensional $T$-supermodule $M$,
we let $\ch(M)$ denote the {\it formal character} of $M$,
that is,
$\ch(M) = \sum_{\lambda \in\Lambda} \dim(M^\lambda) \e^\lambda$.
Set
\[
A(\mu)\,\,:=\,\,\sum_{w\in \W(G,T)} (-1)^{\ell(w)} \, \e^{w\mu},
\]
where $\mu\in\Lambda\tens_\mathz \mathq$ and
$\W(G,T)$ is the Weyl group of $G=\Gev$ with respect to $T$.
As a super-analogue of the Weyl character formula,
we get the following 
description of the Euler characteristic for $\lambda\in\Lambda^+$.

\begin{cor} \label{prp:cor2}
For $\lambda\in\Lambda^+$,
we have
\[
\ch(H^0(\lambda))
\,\,=\,\,
\frac{A(\lambda+\rho_\eve)}{A(\rho_\eve)}
\cdot
\prod_{\delta\in\Delta^+_\odd}(1+\e^{-\delta}).
\]
\end{cor}

\begin{proof}
By the Weyl character formula for $G$,
it is known that the formal character of 
$H^0_\ev(\lambda)$ is given by $A(\lambda+\rho_\eve)/A(\rho_\eve)$,
see \cite[Part~II, 5.10]{Jan03}.
On the other hand,
by definition,
$W^{\mathcal{O}(\U^+)}$ coincides with the dual $(\n^+_\odd)^*$ of $\n^+_\odd$ 
as right $T$-modules.
The formal character of $\wedge(\n^+_\odd)^*$ is given by
$\prod_{\delta\in\Delta^+_\odd}(1+\e^{-\delta})$.
Thus, we are done.
\end{proof}

Set $\rho_\odd:= \frac12 \sum_{\delta\in\Delta_\odd^+}\delta$
and $\rho:=\rho_\eve-\rho_\odd$
in $\Lambda\tens_\mathz\mathq$.
Assume that $w \rho_\odd=\rho_\odd$ for all $w\in \W(G,T)$.
Then for each $\lambda\in\Lambda^+$, we have
\begin{equation}
\ch(H^0(\lambda))
\,\,=\,\,
A(\lambda+\rho)
\cdot
\frac{\prod_{\delta\in\Delta^+_\odd}(\e^{\delta/2}+\e^{-\delta/2})}{
A(\rho_\eve)
}.
\end{equation}
This assumption is satisfied if
$\g=\Lie(\G)$ is a simple Lie superalgebra of type~I 
or $\gl(m|n)$.
For the case of $\G=\GL(m|n)$,
see \cite[Corollary~13.4]{Zub16}.

\subsection{Some Examples} \label{sec:ex}
In this section, we shall see some examples.

\begin{ex} \label{ex:GL(m|n)}
Let us consider the case of $\GL(m|n)$.
In our setting, we shall see Corollary~5.4~in~\cite{Zub06}.
As usual, we choose a split maximal torus $T$ of 
$\GL(m|n)_\ev \cong \GLL_m\times \GLL_n$ so that
\[
T(R)
\,=\, \{\,\big(\mathrm{diag}(x_1,\dots,x_m),\, \mathrm{diag}(x_{m+1},\dots,x_{m+n})\big) 
\in \GLL_m(R)\times\GLL_n(R)\,\},
\]
where $R$ is a commutative algebra.
We may identify $\Lambda = \bigoplus_{i=1}^{m+n} \mathz \lambda_i$.
Then we have $\Lie(\GL(m|n))=\gl(m|n)=\Mat_{m|n}(\Bbbk)$
and
$\W(\GL(m|n)_\ev,T) \cong \mathfrak{S}_m\times \mathfrak{S}_n$,
where $\mathfrak{S}_m$ is the symmetric group on $m$ letters.
The root system of $\GL(m|n)$ with respect to $T$ is 
$\Delta=\{\lambda_i-\lambda_j \in\Lambda \mid 1\leq i\not=j\leq m+n\}$
and 
\[
\Delta_\eve \,\, =\,\,
\{\lambda_i-\lambda_j \mid 1\leq i\not=j\leq m
\,\text{ or }\, m+1\leq i\not=j\leq m+n\}.
\]
Note that, $\Delta=\Delta_\eve \sqcup \Delta_\odd$ (disjoint union).
Define $\gamma:\mathz\Delta\to \mathr$ so that 
$\gamma(\lambda_i):=-i$ for each $1\leq i\leq m+n$.
Then one sees that
$\Delta^{\pm}=\{\pm(\lambda_i-\lambda_j) \mid 1\leq i<j\leq m+n\}$
and the Borel super-subgroup $\B$ (resp. $\B^+$) of $\GL(m|n)$ 
consists of lower (resp. upper) triangular matrices.
In this case, the maximal parabolic super-subgroup of $\GL(m|n)$ is given as
\[
\P(R) \,\, = \,\, \{ \,
\left(\begin{array}{c|c} X & 0 \\ \hline Z & W \end{array}\right) \in \GL(m|n)(R)
\,\}
\]
where $R$ is a commutative superalgebra.
By Theorem~\ref{prp:main}, we see that
\[
\Lambda^\dominant \,\,=\,\,\Lambda^+\,\,=\,\, 
\{ \sum_{i=1}^{m+n} d_i\lambda_i \in\Lambda \mid 
d_1\geq \cdots \geq d_m \text{ and }
d_{m+1}\geq \cdots \geq d_{m+n} \}.
\]
For $\lambda=\sum_{i=1}^{m+n}d_i\lambda_i \in\Lambda^\dominant$,
we shall write $\lambda_+:=\sum_{i=1}^{m}d_i\lambda_i$ and 
$\lambda_-:=\sum_{i=m+1}^{m+n}d_i\lambda_i$.
Then we have 
$H^0_\ev(\lambda)\cong H^0_{\GLL_m}(\lambda_+)\tens H^0_{\GLL_n}(\lambda_-)$,
see \cite[Part~I, Lemma~3.8]{Jan03}.
Here, $H^0_{\GLL_m}(\lambda_+)$ is the induced $\GLL_m$-module for $\Bbbk^{\lambda_+}$.

Let $t_1,\dots,t_{m+n}$ be indeterminates.
As an element of $\mathz[t_1^\pm,\dots,t_{m+n}^\pm]$,
\[
s_{\lambda_+}(t_1,\dots,t_m)
\,\,:=\,\,
\det(t_i^{d_j+m-j})_{1\leq i,j\leq m}\,\, 
\Big/
\prod_{1\leq i<j\leq m}(t_i-t_j)
\]
is the so-called {\it Schur polynomial}.
It is well-known that 
the formal character of $H^0_{\GLL_m}(\lambda_+)$
is given by $A(\lambda_+ +\rho_\eve)/A(\rho_\eve) = 
s_{\lambda_+}(\e^{\lambda_1},\dots,\e^{\lambda_m})$.
Then by Corollary~\ref{prp:cor2},
we see that the formal character of $H^0(\lambda)$ is given as
\[
s_{\lambda_+}(t_1,\dots,t_m)
\cdot s_{\lambda_-}(t_{m+1},\dots ,t_{m+n})
\cdot \underset{m+1\leq j\leq m+n}{\prod_{1\leq i \leq m}}(1+\frac{t_j}{t_i}).
\]
Here, we have replaced each $\e^{\lambda_i}$ with $t_i$.
\qed
\end{ex}

\begin{rem}
We consider the case of $\GL(2|1)$.
Let $T$ be the standard torus of $\Gev$ as before.
If we choose $\gamma:\mathz\Delta\to \mathr$ so that 
$\gamma(\lambda_1)>\gamma(\lambda_3)>\gamma(\lambda_2)$,
then
$\Delta^{-} =
\{ -(\lambda_1-\lambda_2), \, \lambda_2-\lambda_3,\, -(\lambda_1-\lambda_3) \}$
and the corresponding Borel super-subgroup $\B$ of $\GL(2|1)$ is given by
\[
\B(R)\,\,=\,\,
\{
\left(\begin{array}{cc|c}
* & 0 & 0 \\ * & * & * \\ \hline * & 0 & *
\end{array}\right) \in \GL(2|1)(R)
\},
\]
where $R$ is a commutative superalgebra.
In this case, one easily sees that 
$\GL(2|1)$ does not have a maximal parabolic super-subgroup,
that is, there is no such a super-subgroup $\P$ of $\GL(2|1)$ 
satisfying the condition \eqref{eq:parab}.

Recently, Zubkov \cite{Zub16} studied general properties of 
the cohomology 
$H^n(\lambda)=H^n(\GL(m|n)/\B, \mathcal{L}(\Bbbk^\lambda))$
for {\it all} Borel super-subgroups $\B$ of $\GL(m|n)$.
\qed
\end{rem}

For a superalgebra $R$,
the {\it super-transpose} of $A\in\Mat_{m|n}(R)$ is given as
\[
{}^{\mathrm{st}}\!A \,\,:=\,\,
\left(\begin{array}{c|c}
{}^t\! X & {}^t\! Z \\ \hline -{}^t Y & {}^t W
\end{array}\right),
\quad\text{ where }\,\,
A = \left(\begin{array}{c|c} X & Y \\ \hline Z & W \end{array}\right).
\]
Here, ${}^t\! X$ denotes the ordinary transpose of the matrix $X$.

\begin{ex} \label{ex:P(n)}
For $n\geq 2$, 
we consider the following closed super-subgroup $\mathbf{P}(n)$ of $\GL(n|n)$.
For a commutative superalgebra $R$,
\[
\mathbf{P}(n)(R)\,\,:=\,
\{ g\in\GL(n|n)(R) \mid {}^\mathrm{st}g\, J_n\, g = J_n \},
\]
where
$J_n := \left(\begin{array}{c|c} 0 & I_n \\ \hline I_n & 0 \end{array}\right)$
and $I_n$ is the identity matrix of size $n$.
This is a quasireductive supergroup whose even part is $\mathbf{P}(n)_\ev \cong \GLL_n$.

We fix a split maximal torus $T$ of $\mathbf{P}(n)_\ev \cong \GLL_n$ as
the subgroup of all diagonal matrices.
We may identify $\Lambda = \bigoplus_{i=1}^{n} \mathz \lambda_i$ and 
$\W(\mathbf{P}(n)_\ev,T) = \mathfrak{S}_n$.
It is easy to see that the Lie superalgebra of $\mathbf{P}(n)$ is given as
\[
\Lie(\mathbf{P}(n))
\,\,=\,\,
\{ A\in\Mat_{n|n}(\Bbbk) \mid 
{}^\mathrm{st}\!A J_n + J_n A  = 0 \}.
\]
This is the so-called {\it periplectic Lie superalgebra}, 
see \cite[\S2.1.3]{Kac77} and \cite[\S1.1.5]{CheWan12}.
For this reason, we shall call $\mathbf{P}(n)$ the {\it periplectic supergroup}.

Then the root system of $\mathbf{P}(n)$ with respect to $T$ is given as
follows.
\[
\Delta\,\,=\,\,
\{\,
\underbrace{\pm(\lambda_i-\lambda_j)}_{\in\, \Delta_\eve},\,\,
\underbrace{\pm(\lambda_i+\lambda_j),\,\,
2\lambda_p}_{\in\, \Delta_\odd}
\mid 1\leq i<j\leq n,\,\,1\leq p\leq n \}.
\]
Define $\gamma:\mathz\Delta\to \mathr$ so that 
$\gamma(\lambda_i):=-i$ for each $1\leq i\leq n$.
Then we see that 
$\Delta^+ = \{ \lambda_i-\lambda_j,\,\,-(\lambda_i+\lambda_j) \mid 
1\leq i< j\leq n\}$
and
$\Delta^- = \{ -(\lambda_i-\lambda_j),\,\,\lambda_p+\lambda_q \mid 
1\leq i<j\leq n,\,\,1\leq p\leq q\leq n\}$.
In particular, one sees that $-\Delta^+ \not=\Delta^-$.
In this case, the Borel subgroup $B$ (resp. $B^+$) of 
$\mathbf{P}(n)_\ev \cong \GLL_n$
consists of all lower (resp. upper) triangular matrices.
As a closed super-subgroup of $\mathbf{P}(n)(R)$
(where $R$ is a commutative superalgebra),
one sees that
\begin{eqnarray*}
\B(R) 
&=&
\{
\left(\begin{array}{c|c} X & 0 \\ \hline Z & {}^tX^{-1} \end{array}\right)
\mid
X\in B(R_\eve),\,\,
{}^t \! Z X = -{}^t \! X Z
\},
\\
\B^+(R) 
&=&
\{
\left(\begin{array}{c|c} X & Y \\ \hline 0 & {}^t X^{-1} \end{array}\right)
\mid
X\in B^+(R_\eve),\,\,
{}^t \! X^{-1}\, {}^t Y = Y X^{-1}
\},
\\
\P(R)
&=&
\{
\left(\begin{array}{c|c} X & 0 \\ \hline Z & {}^t X^{-1} \end{array}\right)
\mid
X\in \GLL_n(R_\eve),\,\,
{}^t \! Z X = -{}^t \! X Z
\}.
\end{eqnarray*}

By Theorem~\ref{prp:main}, we see that
$\Lambda^\dominant = \Lambda^+ = 
\{ \sum_{i=1}^{n} d_i\lambda_i \in\Lambda \mid 
d_1\geq \cdots \geq d_n \}$.
For $\lambda\in\Lambda^\dominant$,
the formal character of $H^0(\lambda)$ is given as follows.
\[
s_{\lambda}(t_1,\dots,t_n)
\cdot \prod_{1\leq i<j \leq n}(1+\frac{1}{t_i t_j}).
\]
Here, we have replaced each $\e^{\lambda_i}$ with $t_i$.
\qed
\end{ex}

\begin{rem}
As we mentioned before,
the Borel super-subgroup $\B$,
the induced supermodule $H^0(\lambda)$,
the order $\leq$ on $\Lambda$ given in \eqref{eq:order},
and
the notion of {\it maximal} $T$-weight (see Proposition~\ref{prp:maxwt}),
and 
the existence of a maximal parabolic super-subgroup $\P$
{\bf do}~depend on the choice of the homomorphism $\gamma:\mathz\Delta\to \mathr$,
defined in \S\ref{sec:tori,unip}.
\qed
\end{rem}

\appendix

\renewcommand{\theequation}{A.\arabic{equation}}
\setcounter{equation}{0}

\section{Some General Properties of Superspaces} \label{sec:app_mod}

In this appendix,
$\Bbbk$ is a field of characteristic not equal to $2$.

\subsection{Supermodules}

Let $A$ be a superalgebra.
We consider the (ordinary) crossed product algebra $A\rtimes \Bbbk\mathz_2$
via the non-trivial action of the algebra $\Bbbk\mathz_2$ on $A$.
By definition, the multiplication of $A\rtimes \Bbbk\mathz_2$ is given as follows.
\[
(a\rtimes \epsilon)(b\rtimes \eta) \,\, = \,\, 
a(b_0+(-1)^{\epsilon}b_1) \rtimes (\epsilon+\eta),
\]
where $\epsilon,\eta\in\mathz_2$, $a,b=b_0+b_1\in A$ with  
$b_0\in A_\eve$, $b_1\in A_\odd$.

The category of all left $A$-supermodules ${}_A\mathsf{SMod}$ is, by definition,
the category of left $A$- right $\Bbbk\mathz_2$-Hopf modules ${}_A\mathsf{Mod}^{\Bbbk\mathz_2}$.
Since the characteristic of $\Bbbk$ is not $2$,
there is a unique Hopf algebra isomorphism $\Bbbk\mathz_2 \cong (\Bbbk\mathz_2)^*$.
By using the isomorphism, we get the following equivalence of tensor categories.
\begin{equation} \label{eq:SMod->Mod}
{}_A\mathsf{SMod} \overset{\approx}{\longrightarrow} {}_{A\rtimes \Bbbk\mathz_2}\mathsf{Mod};
\quad M\longmapsto M.
\end{equation}
Here, the left $A\rtimes \Bbbk\mathz_2$-module structure on $M$ is given by
\[
(a\rtimes \epsilon) . m \,\, = \,\, a.(m_0 +(-1)^\epsilon m_1),
\]
where $a\in A$, $\epsilon\in\mathz_2$ and $m=m_0+m_1\in M$ with 
$m_0\in M_\eve$, $m_1\in M_\odd$.

A non-zero superalgebra is said to be {\it simple}
if it has no non-trivial two-sided super-ideal.
Here, the notion of a two-sided super-ideal is defined in an obvious way.

\begin{ex} \label{ex:Mat}
For natural numbers $m,n$ and an algebra $R$ (not necessarily commutative),
we let $\Mat_{m|n}(R)$ denote the set of all 
$(m+n)\times (m+n)$ matrices over $R$.
This naturally forms a superalgebra whose grading is given as follows.
\begin{align*}
\Mat_{m|n}(R)_\eve
&=
\{ 
\left(\begin{array}{c|c} X & 0 \\ \hline 0 & W \end{array}\right)
\mid 
X \in\Mat_m(R) ,\,\, W \in \Mat_n(R)\},
\\
\Mat_{m|n}(R)_\odd
&=
\{ 
\left(\begin{array}{c|c} 0 & Y \\ \hline Z & 0 \end{array}\right)
\mid 
Y \in\Mat_{m,n}(R) ,\,\, Z\in \Mat_{n,m}(R)\},
\end{align*}
where $\Mat_m(R)$ (resp. $\Mat_{m,n}(R)$)
is the set of all $m\times m$ (resp. $m\times n$) matrices over $R$.
For simplicity,
we let $\Mat_{m|0}(R)$ denote the ordinary matrix algebra $\Mat_m(R)$.
\qed
\end{ex}

\begin{prop} \label{prp:simple_alg}
Let $A$ be a simple superalgebra.
If $A_\odd \not=0$, then the algebra $A_\eve$ is 
Morita equivalent to $A\rtimes \Bbbk\mathz_2$.
\end{prop}

\begin{proof}
Since $A_\odd A_\odd \oplus A_\odd$ is a non-zero super-ideal of $A$,
the $\mathz_2$-grading of $A$ is {\it strongly graded},
that is, $A_\odd A_\odd = A_\eve$.
This means that $A$ over $A_\eve$ is a {\it right $\Bbbk\mathz_2$-Galois},
see \cite[\S8]{Mon93}.
Then by Hopf-Galois theory, we get the following equivalence.
\begin{equation} \label{eq:Mod->SMod}
{}_{A_\eve} \mathsf{Mod} \overset{\approx}{\longrightarrow} {}_A\mathsf{Mod}^{\Bbbk\mathz_2}
\,(={}_A \mathsf{SMod});
\quad V \longrightarrow A\otimes_{A_\eve} V.
\end{equation}
By combining the equivalence above with \eqref{eq:SMod->Mod},
the claim follows.
\end{proof}

In the situation as above, 
the parity change functor $\Pi: {}_A \mathsf{SMod}\to {}_A \mathsf{SMod}$
can be translated into the following functor.
\[
{}_{A_\eve} \mathsf{Mod} \longrightarrow {}_{A_\eve} \mathsf{Mod};\quad
V\longmapsto A_\odd \otimes_{A_\eve} V.
\]

\subsection{Supercomodules} \label{sec:supercomodules}
Let $C$ be a supercoalgebra.
As a dual of $A\rtimes \Bbbk\mathz_2$,
we can consider the (ordinary) {\it cosmash product} coalgebra
$\Bbbk\mathz_2 \cmdblackltimes C$ of $\Bbbk\mathz_2$ and $C$
which is $\Bbbk\mathz_2\tens C$ as a vector space.
The comultiplication and the couinit of $\Bbbk\mathz_2 \cmdblackltimes C$ 
are respectively given as follows.
\[
\epsilon \cmdblackltimes c \longmapsto 
\sum_c (\epsilon\cmdblackltimes c_{(1)}) \tens 
\big((\epsilon+|c_{(1)}|)\cmdblackltimes c_{(2)} \big),
\quad
\epsilon \cmdblackltimes c \longmapsto \varepsilon_C(c),
\]
where $\epsilon\in\mathz_2$ and $c\in C_\eve \cup C_\odd$.
Here, $\varepsilon_C:C\to\Bbbk$ is the counit of $C$.
Then as a dual of \eqref{eq:SMod->Mod}, we get
\begin{equation} \label{eq:coSMod->Mod}
\SMod^{C} \overset{\approx}{\longrightarrow} \Mod^{\Bbbk\mathz_2 \cmdblackltimes C};
\quad V \longmapsto V.
\end{equation}
Here, the right $\Bbbk\mathz_2 \cmdblackltimes C$-comodule structure on $V$ is given by
\[
v \longmapsto \sum_v v_{(0)} \tens (|v_{(0)}| \cmdblackltimes v_{(1)}),
\]
where $v\in V$ and $v\mapsto \sum_v v_{(0)}\tens v_{(1)}$ is 
the given right $C$-supercomodule structure on $V$.

For a right $C$-supercomodule $V$,
we let $\soc_C(V)$ denote the sum of all simple $C$-super-subcomodules of $V$.
In particular, $\corad(C):=\soc_C(C)$ is called the {\it coradical} of $C$.
The identification \eqref{eq:coSMod->Mod} ensures that
any simple supercomodules are finite dimensional, for example.
Moreover, we get the following lemma.

\begin{lemm} \label{prp:simple_lemma}
Suppose that $C$ is a Hopf superalgebra.
Then for a right $C$-supercomodule $V$,
$V\not=0$ if and only if $\soc_C(V)\not=0$.
\end{lemm}

A Hopf superalgebra $U$ is said to be {\it irreducible}
if $\corad(U)$ is trivial, or equivalently,
the simple $U$-supercomodules are exhausted by the 
purely even or odd trivial $U$-comodule $\Bbbk$,
see \cite[Definition~2]{Mas12}.
For a right $U$-supercomodule $V$,
\[
V^{\co U} \,\, := \,\, 
\{v \in V \mid \sum_v v_{(0)}\tens v_{(1)} = v \tens 1\}
\]
is called the {\it $U$-coinvariant super-subspace} of $V$.

\begin{lemm} \label{prp:simple_lemma2}
Let $V$ be a right $U$-supercomodule.
Then $V^{\co U} = \soc_U(V)$.
In particular, $V\not=0$ if and only if $V^{\co U}\not=0$.
\end{lemm}

\subsection{Cotensors} \label{app:cotens}

Let $C,D$ be supercoalgebras and let $f:C\to D$ be a supercoalgebra map.
For a right $C$-supercomodule $V$ with $\rho:V\to V\tens C$ its structure map,
the map
$\rho|_D : V\overset{\rho}{\longrightarrow} V\tens C 
\overset{\id\tens f}{\longrightarrow} V\tens D$
makes $V$ into a right $D$-supercomodule.
We shall denote it by $\res^C_D(V)$,
called the {\it restriction functor}.
Then $\res^C_D(-)$ becomes a functor from $\SMod^C$ to $\SMod^D$.
For a left $C$-supercomodule $V'$ with $\rho':V'\to C\tens V'$ its structure map,
the {\it cotensor product} $V\cotens_C V'$ of $V$ and $V'$ is given by
\[
V\cotens_C V' 
\,\,:=\,\,
\mathrm{Ker}(
V\tens V' \overset{\rho\tens\id-\id\tens \rho'}{\longrightarrow} V\tens C\tens V'
).
\]
This is a super-subspace of $V\tens V'$.
It is easy to see that $V\cotens_C C\cong V$ and $C\cotens_C V'\cong V'$ as superspaces.
For a Hopf superalgebra $H$, 
the $H$-coinvariant superspace $V^{\co H}$ 
of a right supercomodule $V$ is nothing but $V\cotens_H \Bbbk$,
where $\Bbbk$ is regarded as the trivial right $H$-supercomodule.

We regard $C$ as a left $D$-supercomodule whose structure map is 
given by $(f\tens\id)\circ\Deltaa_C$,
where $\Deltaa_C:C \to C\tens C$ is the comultiplication of $C$.
For a right $D$-supercomodule $W$, 
one easily sees $\ind^C_D(W) := W \cotens_D C$
forms a right $C$-supercomodule
whose structure map is given by $\id\tens \Deltaa_C$.
In this way, we get a (left exact) functor $\ind^C_D(-)$ from $\SMod^D$ to $\SMod^C$,
called the {\it induction functor}.

In the super-situation, the following {\it Frobenius reciprocity} also holds:
For a right $D$-supercomodule $W$ and a left $C$-supercomodule $V$,
we have an isomorphism
\begin{equation} \label{eq:FrobApp}
\underline{\SMod}^D(\res^C_D(V),W) \,\,\overset{\cong}{\longrightarrow} \,\,
\underline{\SMod}^C(V,\ind^C_D(W))
\end{equation}
of superspaces,
which is given by $f\mapsto (f\tens\id_C)\circ \rho$,
where $\rho:V\to V\tens C$ is the $C$-supercomodule structure of $V$. 
Since $\ind^C_D(-)$ is right adjoint to $\res^C_D(-)$,
we see that the functor $\ind^C_D(-)$ preserves injective objects.

\renewcommand{\theequation}{B.\arabic{equation}}
\setcounter{equation}{0}

\section{Clifford Superalgebras} \label{sec:clif}

In this Appendix,
we review some basic facts and notations of Clifford superalgebras
over a field $\Bbbk$ of characteristic not equal to $2$.
Set $\Bbbk^\times:=\Bbbk\setminus\{0\}$.
For details, see 
Wall \cite{Wal64},
Musson \cite[A.3.2]{Mus12} and
Lam \cite{Lam05}.

\subsection{Structures of Clifford Superalgebras}

Let $V$ be a finite-dimensional vector space.
Set $r:=\dim V$.
For a symmetric bilinear form $b:V\times V \to \Bbbk$ on $V$,
the pair $(V,b)$ is called a {\it quadratic space}.
Let $I(V,b)$ be the two-sided ideal of 
the tensor algebra $T(V)$ generated by 
all $xy+yx-b(x,y)$,
where $x,y\in V$.
Set 
\[
\C(V,b) \,\,:=\,\, T(V)/I(V,b).
\]
For simplicity, we sometimes write $\C(V)$ instead of $\C(V,b)$.
This forms a superalgebra over $\Bbbk$ with each element in $V$ regards as odd,
and is called the {\it Clifford superalgebra for $(V,b)$}.
Since the characteristic of $\Bbbk$ is not $2$,
we can choose an orthogonal basis $\{x_1,\dots, x_r\}$ of $V$ with respect to $b$,
that is, a basis of $V$ satisfying $b(x_i,x_j)=0$ if $i\not=j$.
For $1\leq i \leq r$, set
\begin{equation} \label{eq:delta_i}
\delta_i \,\, := \,\, b(x_i,x_i).
\end{equation}
Then one sees that the superalgebra $\C(V,b)$ is generated by 
the odd elements $x_1,\dots,x_r$
with the relations
\[
x_i^2-\delta_i;
\qquad
x_jx_k+x_kx_j,
\]
where $1\leq i\leq r$ and $1\leq j<k\leq r$.
One sees that $\C(V,b)_\epsilon$ is
a $2^{r-1}$-dimensional space for each $\epsilon\in\mathz_2$,
and hence $\C(V,b)$ is $2^r$-dimensional.
If $b$ is trivial (i.e., $b=0$), then $\C(V,b)$ is 
the so-called {\it exterior superalgebra} $\wedge(V)$ on $V$.

The {\it orthogonal sum} $(V,b)\perp(V',b')$ 
(or $V\perp V'$, for short)
of two quadratic spaces $(V,b)$ and $(V',b')$
is a quadratic space
whose underlying space is the direct sum $V\oplus V'$
and the bilinear form is given as follows.
\[
\big((v,v'),\,(w,w')\big) \longmapsto b(v,w) + b'(v',w'),
\]
where $v,w\in V$ and $v',w'\in V'$.
There is an isomorphism of superalgebras
\begin{equation} \label{eq:perp}
\C (V\perp V')
\overset{\cong}{\longrightarrow} 
\C(V) \otimes \C(V')
\end{equation}
given by $(v,v') \mapsto v\otimes 1 + 1 \otimes v'$,
where $v\in V$, $v'\in V'$.

The {\it radical} of the quadratic space $(V,b)$ is defined as follows.
\begin{equation} \label{eq:radical}
\rad(b) 
\,\,:= \,\,
\{ v\in V  \mid b(v,w)=0 \text{ for all } w\in V \}.
\end{equation}

We say that a quadratic space $(V,b)$ is {\it non-degenerate} if $\rad(b)=0$.

\medskip

Suppose that $(V,b)$ is non-degenerate.
We define the {\it signed determinant} of $V$ as follows.
\begin{equation} \label{eq:delta}
\delta \,\, := \,\,
(-1)^{r(r-1)/2} \, \delta_1 \delta_2 \cdots \delta_r.
\end{equation}
It is known that 
the Clifford superalgebra $\C(V,b)$ is central simple over $\Bbbk$
in the category of superalgebras.
In addition, if $r=\dim V$ is even,
then $\C(V,b)$ is still central simple over $\Bbbk$
in the category of algebras.
Moreover, in this case (i.e., $r$ is even), the following structure theorem is known.
For more details, see \cite[Chapter~V, \S2]{Lam05}.
\begin{itemize}
\item
If $\delta\not\in (\Bbbk^\times)^2$,
then the algebra $\C(V,b)_\eve$ is central simple over $\Bbbk(\sqrt{\delta})$.
\item
If $\delta \in (\Bbbk^\times)^2$,
then the superalgebra $\C(V,b)$ is isomorphic to $\Mat_{n|n}(D)$,
where $n\in\mathbb{N}$ and
$D$ is a central division algebra over $\Bbbk$.
\end{itemize}
Here, $\Bbbk(\sqrt{\delta})$ is the quadratic field extension of $\Bbbk$, and
\begin{equation} \label{eq:k^2}
(\Bbbk^\times)^2 \,\,:=\,\, \{ a\in\Bbbk \mid 
a = b^2 \text{ for some } b\in\Bbbk^\times \}.
\end{equation}

\subsection{Simples of Clifford Superalgebras}
In the following, 
we let $(V,b)$ be a quadratic space (not necessarily non-degenerate).
Then there is a non-degenerate subspace $V_s$ of $V$ 
such that $V=\rad(b) \perp V_s$.
Let $\delta$ be the sign determinant of $V_s$.
For simplicity, if $b=0$, then we let $\delta:=0$.
The following fact is well-known,
see Chevalley \cite[Chapter~II, 2.7]{Che97} for example.
For convenience of the reader, we shall give a proof here.

\begin{lemm} \label{prp:Cliff-Jac}
The (ordinary) Jacobson radical $J$ of the algebra $\C(V)$ coincides with
the two-sided ideal of $\C(V)$ generated by $\rad(b)$,
and the quotient (ordinary) algebra $\C(V)/J$ is isomorphic to $\C(V_s)$.
\end{lemm}

\begin{proof}
For simplicity,
we let $\mathfrak{r}$ be the two-sided ideal of $\C(V)$ generated by $\rad(b)$.
We shall identify $\C(V_s)$ as a subalgebra of $\C(V)$
via the isomorphism $\C(V)\cong \C(\rad(b))\tens \C(V_s)$.
Then $\mathfrak{r}+\C(V_s)$ forms a subalgebra of $\C(V)$.
Since $V=\rad(b)\perp V_s$,
we see that $\mathfrak{r}+\C(V_s) = \C(V)$, and hence we have
\[
\C(V)/\mathfrak{r} \,\, \cong \,\,
\C(V_s) / (\mathfrak{r} \cap \C(V_s)).
\]

Fix $x\in\rad(b)$.
For each $u=u_0+u_1\in\C(V)$ with 
$u_0\in \C(V)_\eve$ and $u_1\in\C(V)_\odd$,
we see that
$(xu)^2 = xuxu = x^2(u_0-u_1)u = 0$, since $x^2=0$.
Thus, we get that $1-xu$ is invertible in $\C(V)$, and hence $\rad(b) \subseteq J$.
In particular, we have $\mathfrak{r} \subseteq J$.
Thus, $\mathfrak{r} \cap \C(V_s)$ is a nil-ideal of $\C(V_s)$.
Since the Jacobson radical of $\C(V_s)$ is trivial,
we get $\C(V)/\mathfrak{r}\cong \C(V_s)$
and $\mathfrak{r}=J$.
\end{proof}

To find simple supermodules,
the following lemma is useful.

\begin{lemm} \label{prp:Clif-lem}
There exists a $(1+\dim V)$-dimensional quadratic space $(V',b')$
such that the category of $\C(V)$-supermodules
is equivalent to 
the category of $\C(V')$-modules.
\end{lemm}

\begin{proof}
For simplicity, we set $A:=\C(V,b)$.
Let $\Bbbk x$ be a one-dimensional vector space over $\Bbbk$ with a base $x$,
and let $b_x: \Bbbk x \times \Bbbk x \to\Bbbk$ 
be a symmetric bilinear form defined by $b_x(x,x)=1$.
Then we have the following isomorphism of (ordinary) algebras.
\[
A\rtimes \Bbbk\mathz_2 
\overset{\cong}{\longrightarrow}
A\otimes \C(\Bbbk x, b_x);
\quad
a\rtimes \epsilon \longmapsto a\otimes x^\epsilon.
\]
where $a\in A$ and $\epsilon\in\mathz_2$.
Here, the tensor product $A \otimes \C(\Bbbk x ,b_x)$ is $\mathz_2$-graded, as before.
Set $(V',b') := (V,b) \perp (\Bbbk x, b_x)$.
Then we have an isomorphism of (ordinary) algebras
$A\rtimes \Bbbk\mathz_2  \,\, \cong \,\, \C (V',b')$.
Hence, by the equivalence \eqref{eq:SMod->Mod},
the category of all left $A$-supermodules 
can be identified with
the category of all left $\C(V',b')$-modules.
\end{proof}

\begin{prop} \label{prp:Clif-simples}
There is a unique simple left $\C(V,b)$-supermodule $\mathfrak{u}$
up to isomorphism and parity change.
If (1) $\delta=0$ or (2) $\dim V$ is even and $\delta\in (\Bbbk^\times)^2$,
then $\mathfrak{u} \not= \Pi \mathfrak{u}$.
Otherwise, $\mathfrak{u} = \Pi \mathfrak{u}$.
\end{prop}

\begin{proof}
First, suppose that $\delta=0$.
Then $\C(V,b)$ coincides with the 
exterior superalgebra $\wedge(V)$ on $V$.
It is easy to see that
$\wedge(V)$ has a unique one-dimensional simple supermodule 
$\mathfrak{u}$ up to isomorphism
which satisfies
$\mathfrak{u}\not\cong \Pi \mathfrak{u}$.

Next, suppose that $\delta \in \Bbbk^\times$.
For simplicity, we set $A:=\C(V,b)$ and $r:=\dim V$.
By Lemma~\ref{prp:Cliff-Jac},
it is enough to consider the case when $(V,b)$ is non-degenerate.
(i)~Assume that $r$ is odd.
We use the notation in Lemma~\ref{prp:Clif-lem}.
Note that, $V'$ is also non-degenerate.
Since $\dim V' = r+1$ is even,
we see that $\C(V')$ is (Artinian) simple as an ordinary algebra.
Thus, 
$A$ has a unique simple supermodule $\mathfrak{u}$ up to isomorphism
which satisfies $\Pi\mathfrak{u} = \mathfrak{u}$.
(ii)~Assume that $r$ is even and $\delta\not\in (\Bbbk^\times)^2$.
In this case, $A_\eve$ is (Artinian) simple.
By Proposition~\ref{prp:simple_alg},
$A\rtimes \Bbbk\mathz_2$ is Morita equivalent to $A_\eve$.
Thus, $A$ has a unique simple supermodule $\mathfrak{u}$ up to isomorphism
which satisfies $\Pi\mathfrak{u} = \mathfrak{u}$,
by using the equivalence \eqref{eq:SMod->Mod}.
(iii)~Assume that $r$ is even and $\delta\in(\Bbbk^\times)^2$.
Then the superalgebra $A$ can be identified with
$\Mat_{n|n}(D)$ for
some $n\in\mathbb{N}$ and a central division algebra $D$ over $\Bbbk$.
Thus, in the following, we identify $A_\eve$ with
\[
\Mat_{n|n}(D)_\eve=
\{ 
\left(\begin{array}{c|c} X & 0 \\ \hline 0 & W \end{array}\right)
\mid 
X,W \in\Mat_n(D)\}.
\]
As column vectors of length $2n$ with entries in $D$, we set
\[
e\,\, :=\,\, {}^t(\,\underbrace{1,\dots,1}_{n\text{-times}},{0,\dots,0}\,),
\qquad
e'\,\,:=\,\, {}^t(\,\underbrace{0,\dots,0}_{n\text{-times}},{1,\dots,1}\,).
\]
Then $M:=A_\eve\, e$ and $M':=A_\eve\, e'$ are the distinct simple left $A_\eve$-modules.
By Proposition~\ref{prp:simple_alg},
the corresponding simple left $A$-supermodules are given by
$\mathfrak{u} := A \otimes_{A_\eve} M$ and $\mathfrak{u}' := A \otimes_{A_\eve} M'$.
One easily sees that $\Pi \mathfrak{u}=\mathfrak{u}'$,
and hence $\mathfrak{u} \not= \Pi \mathfrak{u}$.
\end{proof}

\subsection{The Algebraically Closed Case}

Assume that the base field $\Bbbk$ is algebraically closed.
In this section,
we shall calculate the dimensions of simple supermodules
of Clifford superalgebras.

Let $(V,b)$ be a quadratic space over $\Bbbk$, in general.
The quotient space $\overline{V}:=V/\rad(b)$ is non-degenerate.
Set $d:=\dim\overline{V}$.
By Proposition~\ref{prp:Clif-simples}, 
$\C(V,b)$ has a unique simple supermodule $\mathfrak{u}$
up to isomorphism and parity change.

\begin{prop} \label{prp:dim(u)}
The dimension of the vector space $\mathfrak{u}$ is given by $2^{\lfloor (d+1)/2 \rfloor}$,
where $\lfloor (d+1)/2 \rfloor$ is the largest integer less than or equal to
$(d+1)/2$.
\end{prop}

\begin{proof}
First, we suppose that the signed determinant $\delta$ of $V$, 
defined in \eqref{eq:delta}, is zero.
In this case, $\dim\mathfrak{u}=1$ and $d=0$,
and hence the claim follows.
Next, we suppose that $\delta\not=0$ and $d$ is odd.
In this case, the algebra
$\C(\overline{V},\overline{b})\rtimes \Bbbk\mathz_2$ is central simple over $\Bbbk$.
Since $\Bbbk$ is algebraically closed,
one sees that the algebra $\C(\overline{V},\overline{b})\rtimes \Bbbk\mathz_2$ is 
isomorphic to $\Mat_{m}(\Bbbk)$,
where $m=2^{(d+1)/2}$.
Hence, we are done.
Finally, we suppose that $\delta\not=0$ and $d$ is even.
Note that, $\lfloor (d+1)/2 \rfloor = d/2$.
Since $\Bbbk$ is algebraically closed, 
the element $\delta$ ($\not=0$) is always in $(\Bbbk^\times)^2 = \Bbbk^\times$.
Thus,
the algebra $\C(\overline{V},\overline{b})_\eve$ is isomorphic to 
$\Mat_n(\Bbbk)\times \Mat_n(\Bbbk)$,
where $n=2^{d/2}$.
Hence, we have $\dim\mathfrak{u}=n=2^{d/2}$.
\end{proof}

A {\it maximal totally isotropic subspace} $V^\dagger$ of $(V,b)$ 
is a maximal subspace $W$ of $V$ satisfying $b(W,W)=0$.
It is known that the {\it Witt index} 
$\Ind(V) := \lfloor d/2 \rfloor$ of $(V,b)$
coincides with the dimension of the quotient space $V^\dagger/\rad(b)$,
see \cite[Chapter~I, Corollary~4.4]{Lam05} for example.
Therefore,
we obtain $\dim V - \dim V^\dagger = \lfloor (d+1)/2 \rfloor$.
Here, we used $d = \lfloor d/2 \rfloor + \lfloor (d+1)/2 \rfloor$.
Then by Proposition~\ref{prp:dim(u)},
we get
\begin{equation} \label{eq:dim(u)}
\dim \mathfrak{u} \,\, = \,\,
2^{\dim V - \dim V^\dagger}.
\end{equation}
This is another description of the dimension of $\mathfrak{u}$.

\medskip

\section*{Acknowledgements}
Almost all of the manuscript was prepared while 
the author was a 
Pacific Institute for the Mathematical Sciences (PIMS) 
Postdoctoral Fellow position at University of Alberta.
The author would like to thank PIMS for their hospitality and financial support.

\medskip


\begin{thebibliography}{10}

\bibitem{BicRic16}
Julien Bichon and Simon Riche, \emph{Hopf algebras having a dense big cell},
  Trans. Amer. Math. Soc. \textbf{368} (2016), no.~1, 515--538.

\bibitem{Bru06}
Jonathan Brundan, \emph{Modular representations of the supergroup {$Q(n)$}.
  {II}}, Pacific J. Math. \textbf{224} (2006), no.~1, 65--90.

\bibitem{BruKle03}
Jonathan Brundan and Alexander Kleshchev, \emph{Modular representations of the
  supergroup {$Q(n)$}. {I}}, J. Algebra \textbf{260} (2003), no.~1, 64--98,
  Special issue celebrating the 80th birthday of Robert Steinberg.

\bibitem{BruKuj03}
Jonathan Brundan and Jonathan Kujawa, \emph{A new proof of the {M}ullineux
  conjecture}, J. Algebraic Combin. \textbf{18} (2003), no.~1, 13--39.

\bibitem{CheShuWan18}
Shun-Jen {Cheng}, Bin {Shu}, and Weiqiang {Wang}, \emph{Modular representations
  of exceptional supergroups}, to appear in Math. Z. (2018), preprint
  {\rm{\texttt{arXiv:}}}{\bf 1801.05077}~[math.RT].

\bibitem{CheWan12}
Shun-Jen Cheng and Weiqiang Wang, \emph{Dualities and representations of {L}ie
  superalgebras}, Graduate Studies in Mathematics, vol. 144, American
  Mathematical Society, Providence, RI, 2012. \MR{3012224}

\bibitem{Che97}
Claude Chevalley, \emph{The algebraic theory of spinors and {C}lifford
  algebras}, Collected works. Vol. 2, Springer-Verlag, Berlin, 1997.

\bibitem{FioGav12}
Rita Fioresi and Fabio Gavarini, \emph{Chevalley supergroups}, Mem. Amer. Math.
  Soc. \textbf{215} (2012), no.~1014, vi+64.

\bibitem{FioGav13}
\bysame, \emph{Algebraic supergroups with {L}ie superalgebras of classical
  type}, J. Lie Theory \textbf{23} (2013), no.~1, 143--158. \MR{3060770}

\bibitem{Jan03}
Jens~Carsten Jantzen, \emph{Representations of algebraic groups}, second ed.,
  Mathematical Surveys and Monographs, vol. 107, American Mathematical Society,
  Providence, RI, 2003.

\bibitem{Kac77}
Victor~G. Kac, \emph{Lie superalgebras}, Advances in Math. \textbf{26} (1977),
  no.~1, 8--96.

\bibitem{Kac77-3}
\bysame, \emph{Representations of classical {L}ie superalgebras}, Differential
  geometrical methods in mathematical physics, {II} ({P}roc. {C}onf., {U}niv.
  {B}onn, {B}onn, 1977), Lecture Notes in Math., vol. 676, Springer, Berlin,
  1978, pp.~597--626. \MR{519631}

\bibitem{Lam05}
Tsit~Yuen Lam, \emph{Introduction to quadratic forms over fields}, Graduate
  Studies in Mathematics, vol.~67, American Mathematical Society, Providence,
  RI, 2005.

\bibitem{Mar18}
Franti\v{s}ek Marko, \emph{Irreducibility of induced supermodules for general
  linear supergroups}, J. Algebra \textbf{494} (2018), 92--110. \MR{3723172}

\bibitem{MarZub18}
Franti\v{s}ek Marko and Alexandr~N. Zubkov, \emph{Blocks for the general linear
  supergroup {${\rm GL}(m|n)$}}, Transform. Groups \textbf{23} (2018), no.~1,
  185--215. \MR{3763946}

\bibitem{Mas05}
Akira Masuoka, \emph{The fundamental correspondences in super affine groups and
  super formal groups}, J. Pure Appl. Algebra \textbf{202} (2005), 284--312.

\bibitem{Mas12}
\bysame, \emph{Harish-{C}handra pairs for algebraic affine supergroup schemes
  over an arbitrary field}, Transform. Groups \textbf{17} (2012), no.~4,
  1085--1121.

\bibitem{MasShi17}
Akira Masuoka and Taiki Shibata, \emph{Algebraic supergroups and
  {H}arish-{C}handra pairs over a commutative ring}, Trans. Amer. Math. Soc.
  (2017), no.~369, 3443--3481.

\bibitem{MasShi18}
\bysame, \emph{On functor points of affine supergroups}, J. Algebra (2018),
  no.~503, 534--572.

\bibitem{MasTak18}
Akira Masuoka and Yuta Takahashi, \emph{{Geometric construction of quotients
  $G/H$ in supersymmetry}}, preprint {\rm{\texttt{arXiv:}}} \textbf{1808.05753}
  (2018), [math.AG].

\bibitem{MasZub11}
Akira Masuoka and Alexandr~N. Zubkov, \emph{Quotient sheaves of algebraic
  supergroups are superschemes}, J. Algebra \textbf{348} (2011), 135--170.

\bibitem{Mon93}
Susan Montgomery, \emph{Hopf algebras and their actions on rings}, CBMS
  Regional Conference Series in Mathematics, vol.~82, Published for the
  Conference Board of the Mathematical Sciences, Washington, DC; by the
  American Mathematical Society, Providence, RI, 1993.

\bibitem{Mus12}
Ian~M. Musson, \emph{Lie superalgebras and enveloping algebras}, Graduate
  Studies in Mathematics, vol. 131, American Mathematical Society, Providence,
  RI, 2012.

\bibitem{ParWan91}
Brian Parshall and Jian~Pan Wang, \emph{Quantum linear groups}, Mem. Amer.
  Math. Soc. \textbf{89} (1991), no.~439, vi+157.

\bibitem{Pen88-1}
Ivan Penkov, \emph{Borel-{W}eil-{B}ott theory for classical {L}ie supergroups},
  Current problems in mathematics. {N}ewest results, {V}ol. 32, Itogi Nauki i
  Tekhniki, Akad. Nauk SSSR, Vsesoyuz. Inst. Nauchn. i Tekhn. Inform., Moscow,
  1988, Translated in J. Soviet Math. {{\bf{5}}1} (1990), no. 1, 2108--2140,
  pp.~71--124. \MR{957752}

\bibitem{PenSer92}
Ivan Penkov and Vera Serganova, \emph{Representations of classical {L}ie
  superalgebras of type {${\rm I}$}}, Indag. Math. (N.S.) \textbf{3} (1992),
  no.~4, 419--466. \MR{1201236}

\bibitem{Ser11}
Vera Serganova, \emph{Quasireductive supergroups}, New developments in {L}ie
  theory and its applications, Contemp. Math., vol. 544, Amer. Math. Soc.,
  Providence, RI, 2011, pp.~141--159.

\bibitem{Serg01}
Alexander Sergeev, \emph{An analog of the classical invariant theory for {L}ie
  superalgebras. {I}, {II}}, Michigan Math. J. \textbf{49} (2001), no.~1,
  113--146, 147--168. \MR{1827078}

\bibitem{ShuWan08}
Bin Shu and Weiqiang Wang, \emph{Modular representations of the
  ortho-symplectic supergroups}, Proc. Lond. Math. Soc. (3) \textbf{96} (2008),
  no.~1, 251--271.

\bibitem{Wal64}
Charles Terence~Clegg Wall, \emph{Graded {B}rauer groups}, J. Reine Angew.
  Math. \textbf{213} (1963/1964), 187--199.

\bibitem{Wei09}
Rainer Weissauer, \emph{Semisimple algebraic tensor categories}, preprint
  {\rm{\texttt{arXiv:}}} \textbf{0909.1793} (2009), [math.CT].

\bibitem{Zub06}
Alexandr~N. Zubkov, \emph{Some properties of general linear supergroups and of
  {S}chur superalgebras}, Algebra Logika \textbf{45} (2006), no.~3, 257--299,
  375.

\bibitem{Zub16}
\bysame, \emph{Some homological properties of {${\rm GL}(m|n)$} in arbitrary
  characteristic}, J. Algebra Appl. \textbf{15} (2016), no.~7, 1650119, 26.
  \MR{3528547}

\end{thebibliography}

\providecommand{\bysame}{\leavevmode\hbox to3em{\hrulefill}\thinspace}
\providecommand{\MR}{\relax\ifhmode\unskip\space\fi MR }
\providecommand{\MRhref}[2]{%
  \href{http://www.ams.org/mathscinet-getitem?mr=#1}{#2}
}
\providecommand{\href}[2]{#2}

\end{document}